\documentclass{article} 
\usepackage{graphicx}
\usepackage{amsmath}
\usepackage{amssymb,amsthm}
\usepackage[colorlinks=true]{hyperref}
\usepackage{pdfsync}

\newtheorem{theorem}{Theorem}

\newtheorem{lemma}{Lemma}
\newtheorem{example}{Example}
\theoremstyle{definition}\newtheorem{remark}{Remark}

\usepackage{graphics} 
\usepackage{epsfig} 
\usepackage{epstopdf}
\usepackage{subfigure}
\usepackage{color}

\graphicspath{{./Figures/}}

\allowdisplaybreaks[4]

\newcounter{cst}

\title{Boundary control of heat-heat cascades}

\author{Hugo Lhachemi, Christophe Prieur, and Emmanuel Tr{\'e}lat
\thanks{Hugo Lhachemi is with Universit{\'e} Paris-Saclay, CNRS, CentraleSup{\'e}lec, Laboratoire des signaux et syst\`emes, 91190, Gif-sur-Yvette, France (email: hugo.lhachemi@centralesupelec.fr).}
\thanks{Christophe Prieur is with Universit\'e Grenoble Alpes, CNRS, Gipsa-lab, 38000 Grenoble, France (e-mail: christophe.prieur@gipsa-lab.fr).}
\thanks{Emmanuel Trélat is with Sorbonne Universit\'e, Universit\'e Paris Cit\'e, CNRS, Inria, Laboratoire Jacques-Louis Lions, LJLL, F-75005 Paris, France (e-mail: emmanuel.trelat@sorbonne-universite.fr).}}

\date{}

\begin{document}

\maketitle

\begin{abstract}                        
This paper addresses the problem of feedback stabilization of a cascade of two heat equations that are coupled in the boundary conditions, the input being a boundary control for the first component of the cascade. Two distinct control input settings are studied: one being collocated with the coupling condition of the two heat equations, and the other being non-collocated. These two different configurations induce different controllability properties.
\\
The key idea developed in this paper is to carry out spectral reductions, not for each of the two components of the cascade separately, but instead, directly for the PDE cascade viewed as one single system. A detailed study of the eigenelements of the PDE cascade yields a complete characterization of the spectral mode controllability and allows us to derive an explicit state-feedback control strategy for the exponential stabilization of the plant. This approach is extended to a systematic output-feedback control strategy either with a distributed output operator or with a pointwise measurement done on the second heat equation of the PDE cascade. In both state-feedback and output-feedback scenarios, stabilization results are established in $L^2$ and $H^1$ norms. Finally, we show how the results developed in this paper for the two studied heat-heat cascades extend to their dual problems.
\end{abstract}

\section{Introduction}\label{sec: Introduction}

The design of control strategies for partial differential equations (PDEs) cascades remains challenging. Backstepping design methods have been proposed in \cite{chen2017backstepping,ghousein2020backstepping} for coupled hyperbolic-parabolic PDE systems. Control design for a heat-wave PDE with coupling at the boundary has been studied in \cite{zhang2004polynomial}. Boundary null controllability of two coupled parabolic PDEs has been studied in \cite{bhandari2021boundary}. Hyperbolic-elliptic couplings have been reported in \cite{chowdhury2023boundary,rosier2013unique}. The output-feedback stabilization of a wave-heat cascade is reported in~\cite{lhachemi2025controllability} by means of spectral reduction methods.
\\[2mm]
This paper focuses on the topic of stabilization of heat-heat cascades. While approximate and null controllability properties of such systems have been studied for a number of heat-heat cascade configurations in the literature  (see \cite{ammar2011recent,benabdallah2020block,boyer2022controllability,khodja2016new}), it seems that only few references address the actual design of feedback control strategies. The state-feedback stabilization in $H^{-1} \times H^1$-norm of a heat-heat cascade whose first component is open-loop stable with null reaction coefficient is addressed in~\cite{wang2015stabilization}. The case of a cascade of two heat equations with identical reaction coefficients is studied in \cite{kang2016stabilisation}. Assuming that two Dirichlet measurements are available on each of the two components of the PDE cascade, the authors design an output-feedback control strategy using backstepping transformations that stabilize the plant in $L^2$-norm. 
In \cite{tang2024boundary,tang2025sampled}, the authors consider the following cascade of two heat equations with distinct reaction coefficients:
\begin{subequations}\label{eq: system}
	\begin{align}
		& \partial_t y = \partial_{xx} y + a y , \label{eq: system - y} \\
		& \partial_t z = \partial_{xx} z + b z , \label{eq: system - z} \\
		& \partial_x y(t,0) = s z(t,0) , && y(t,1) = \partial_x z(t,1) = 0 , \\
		& \partial_x z(t,0) = u(t) , \\
		& y(0,x) = y_0(x) , && z(0,x) = z_0(x),
	\end{align}
\end{subequations}
for $t > 0$ and $x \in (0,1)$, where $a,b,s\in\mathbb{R}$ with $s\neq 0$, $u(t)$ is the control input at time $t$, and $y_0,z_0 \in L^2(0,1)$ are given initial conditions. In \cite{tang2024boundary,tang2025sampled}, a state-feedback control strategy is designed for \eqref{eq: system}, based on spectral reduction methods. Under an assumption on the reaction coefficients $a,b$, which appears to be required by the control method used for decoupling parts of the dynamics, but which is not related to the intrinsic controllability properties of the original plant \eqref{eq: system}, sufficient stability conditions are derived for stabilization in $L^2$-norm. 
\\[2mm]
In this paper, we first design a state-feedback control strategy for the exponential stabilization of \eqref{eq: system}. The adopted approach relies on spectral reduction methods~\cite{coron2004global,lhachemi2020pi,russell1978controllability}. However, contrarily to \cite{tang2025sampled} where the spectral reduction of \eqref{eq: system} is done on each individual component of the PDE cascade, yielding a cascade connection taking the form of a series, we perform the spectral reduction by considering the PDE cascade \eqref{eq: system} as one single system~\cite{lhachemi2025controllability}. Our analysis relies on a detailed study of the eigenelements of the PDE cascade, showing that the generalized eigenvectors of the underlying unbounded operator forms a Riesz basis. The benefit of our approach is that it allows a full characterization of the controllability properties of each individual mode of the PDE cascade, and in turn, of the exact controllability properties of the system in a suitable Hilbert space (see Appendix~\ref{appendix}), and the derivation of a systematic state-feedback control strategy for the exponential stabilization of the plant. We extend our approach to a systematic output-feedback control strategy either with a distributed output operator or with a pointwise measurement done on the second heat equation of the PDE cascade. In particular, we establish the following result that, at this stage, we state in an informal way (see Theorems~\ref{thm2} and~\ref{thm3} for precise and complete statements).

\begin{theorem}
Under (explicit, necessary and sufficient) controllability and observability conditions, 
the closed-loop controlled system \eqref{eq: system} can be exponentially stabilized at any decay rate in $\mathcal{H}^0 = L^2(0,1) \times L^2(0,1)$ and also in $\mathcal{H}^1 = H^1(0,1) \times H^1(0,1)$, with an explicit output-feedback control. 
\end{theorem}

It is worth mentioning that the approach developed in this paper also allows to handle the system 
\begin{subequations}\label{eq: system 2}
	\begin{align}
		& \partial_t y = \partial_{xx} y + a y , \\
		& \partial_t z = \partial_{xx} z + b z , \\
		& \partial_x y(t,0) = s z(t,0) , && y(t,1) = \partial_x z(t,0) = 0 , \\
		& \partial_x z(t,1) = u(t) , \\
		& y(0,x) = y_0(x) , && z(0,x) = z_0(x),
	\end{align}
\end{subequations}
for $t > 0$ and $x \in (0,1)$. Compared with \eqref{eq: system} in which the control input $u(t)$ and the coupling between the two heat equations are collocated at $x=0$ for the $z$-equation, in \eqref{eq: system 2} the control and the coupling are placed at different boundaries. The PDE cascade \eqref{eq: system 2} was studied in \cite{kang2016stabilisation} in the particular case $a=b$. As we shall find out (see in particular Remark~\ref{rem: system 2}), while all modes of \eqref{eq: system 2} are controllable, in the collocated case \eqref{eq: system} controllability is lost for some specific modes for particular choices of the reaction coefficients $a,b\in\mathbb{R}$. This is why, in this paper, we focus on the system \eqref{eq: system} and we only give a succinct presentation of our results for \eqref{eq: system 2} through some remarks.
\\[2mm]
The paper is organized as follows. First, we make a spectral analysis of the PDE cascades \eqref{eq: system} and \eqref{eq: system 2} in Section~\ref{sec: spectral analysis}, establishing in particular the Riesz basis property. Then, in Section~\ref{sec: spectral reduction}, we perform a spectral reduction of the problem and we establish spectral controllability properties. Our state-feedback control strategy is developed in Section~\ref{sec: state-feedback} and its extension to output-feedback is done in Section~\ref{sec: output-feedback}. The applicability of our results to the two dual PDE cascade problems is assessed in Section~\ref{sec: Dual problems}. Concluding remarks are formulated in Section~\ref{sec: conclusion}.

\section{Spectral analysis}\label{sec: spectral analysis}

We consider the Hilbert space 
\begin{subequations}\label{eq: space H0}
\begin{equation}
	\mathcal{H}^0 = L^2(0,1) \times L^2(0,1)
\end{equation}
endowed with the inner product
\begin{equation}
	\langle (f_1,g_1) , (f_2,g_2) \rangle = \int_0^1 (f_1 f_2 + g_1 g_2) \,\mathrm{d}x 
\end{equation}
\end{subequations}
and associated norm $\Vert \cdot \Vert_{\mathcal{H}_0}$.

\subsection{Computation of the spectrum and Riesz basis property}

In view of the studied systems \eqref{eq: system}-\eqref{eq: system 2}, we define the operator
\begin{equation*}
	\mathcal{A}(f,g) = (f''+af,g''+bg)
\end{equation*}
of domain
\begin{equation*}
	D(\mathcal{A}) = \big\{ (f,g) \in H^2(0,1) \times H^2(0,1) \,\mid\, 		
	f(1)=g'(0)=g'(1)=0 ,\, f'(0) = s g(0) \big\} . 
\end{equation*}
Let us describe the eigenelements of the operator $\mathcal{A}$.
We define the sets
	\begin{subequations}
	\begin{align}
		\Delta(a,b) & = \big\{ (n,m)\in\mathbb{N}^2 \,\mid\,  4(b-a) = (4 m^2 - (2n+1)^2 ) \pi^2 \big\} , \label{eq: def Delta(a,b)} \\
		\Delta_1(a,b) & = \left\{ n\in\mathbb{N} \,\mid\, \exists m\in\mathbb{N} \;\mathrm{s.t.}\; (n,m)\in\Delta(a,b) \right\}  , \nonumber\\
		\Delta_2(a,b) & = \left\{ m\in\mathbb{N} \,\mid\, \exists n\in\mathbb{N} \;\mathrm{s.t.}\; (n,m)\in\Delta(a,b) \right\} . \nonumber
	\end{align}
	\end{subequations}

\begin{lemma}\label{lem: eigenelements A}
The family $\Phi = \{ \phi_{1,n} \,\mid\, n\geq 0 \} \cup \{ \phi_{2,m} \,\mid\, m \geq 0\}$ forms a set of generalized eigenvectors of $\mathcal{A}$ associated with the eigenvalues $\Lambda = \{ \lambda_{1,n} \,\mid\, n\geq 0 \} \cup \{ \lambda_{2,m} \,\mid\, m \geq 0\}$, where
	\begin{subequations}\label{eq: point spectrum}
	\begin{align}
		\lambda_{1,n} & = a - (n+1/2)^2\pi^2 , && n \geq 0 , \\
		\lambda_{2,m} & = b - m^2\pi^2 , && m \geq 0 ,
	\end{align}
	\end{subequations}
	$\phi_{1,n} = (\phi_{1,n}^1,\phi_{1,n}^2)$ for every $n \geq 0$, with
	\begin{equation*}
		\phi_{1,n}^1(x) = \sqrt{2} \cos\left( (n+1/2) \pi x \right) , \quad \phi_{1,n}^2(x) = 0,
	\end{equation*}
	$\phi_{2,m} = (\phi_{2,m}^1,\phi_{2,m}^2)$ for every $m \geq 0$, with
	$\phi_{2,m}^1$ defined as follows:
	\begin{subequations}
	\begin{itemize}
		\item For $m \notin \Delta_2(a,b)$ with $b-a = m^2 \pi^2$:
		\begin{equation}\label{eq: phi_2_m^1 - 1}
		\phi_{2,m}^1(x) = s \mu_m (-1+x) .
		\end{equation}
		\item For $m \notin \Delta_2(a,b)$ with $b-a \neq m^2 \pi^2$:
		\begin{equation}\label{eq: phi_2_m^1 - 2}
		\phi_{2,m}^1(x) = -\frac{s \mu_m}{r_m \cosh(r_m)}\sinh(r_m(1-x))
		\end{equation}
		where $r_m$ is one of the two square roots of $\lambda_{2,m}-a=b-a-m^2\pi^2$. Note that this function is real-valued.
		\item For $m \in \Delta_2(a,b)$, denoting by $k \in \mathbb{N}$ the integer such that $(k,m)\in\Delta(a,b)$:
		\begin{equation}\label{eq: phi_2_m^1 - 3}
		\phi_{2,m}^1(x) =  \frac{s \mu_m}{(k+1/2)\pi} (1-x) \sin((k+1/2)\pi x) ,
		\end{equation}
	\end{itemize}
	\end{subequations}
	\begin{equation*}
		\phi_{2,m}^2(x) = 
		\left\{
		\begin{array}{ll}
			1 & \quad \textrm{if}\ \ m = 0 , \\
			\sqrt{2} \cos\left( m \pi x \right)  & \quad \textrm{if}\ \ m \geq 1 ,
		\end{array}
		\right.
	\end{equation*}
	where $\mu_m = \phi_{2,m}^2(0)$; we have $\mu_m = 1$ if $m=0$ and $\mu_m = \sqrt{2}$ if $m \geq 1$.
	Moreover: 
	\begin{enumerate}	
		\item $\lambda_{1,n}$ is an eigenvalue of $\mathcal{A}$ with associated eigenvector $\phi_{1,n}$:
		\begin{equation}\label{eq: A phi_1_n}
			\mathcal{A} \phi_{1,n} = \lambda_{1,n} \phi_{1,n}.
		\end{equation}			
		\item If $m\notin\Delta_2(a,b)$ then $\lambda_{2,m}$ is an eigenvalue of $\mathcal{A}$ with associated eigenvector $\phi_{2,m}$:
		\begin{equation}\label{eq: A phi_2_m}
			\mathcal{A} \phi_{2,m} = \lambda_{2,m} \phi_{2,m} .
		\end{equation} 
		\item $(n,m)\in\Delta(a,b)$ if and only if $\lambda = \lambda_{1,n} = \lambda_{2,m}$. In that case $\lambda$ is an eigenvalue of $\mathcal{A}$ of geometric multiplicity $1$ but of algebraic mutiplicity $2$, and
		\begin{subequations}\label{eq: phi generalized eigenvector}
		\begin{align}
			\mathcal{A} \phi_{2,m} & = \lambda \phi_{2,m} - \sqrt{2} s \mu_m \phi_{1,n} , \label{eq: phi generalized eigenvector - 1} \\
			\mathcal{A} \phi_{1,n} & = \lambda \phi_{1,n} . \label{eq: phi generalized eigenvector - 2}
		\end{align}		
		\end{subequations}	
		\item $\Phi$ contains the finite number $\mathrm{Card}(\Delta(a,b))$ of generalized eigenvectors that are not eigenvectors of $\mathcal{A}$.
	\end{enumerate}	 
\end{lemma}

\begin{proof}
	Solving $\mathcal{A}(f,g) = \lambda (f,g)$ for some $(f,g)\in D(\mathcal{A})$ and $\lambda\in\mathbb{C}$ gives:
	\begin{equation*}
		\left\{ \begin{array}{l} f''+af=\lambda f \\ f(1)=0 , \; f'(0) = s g(0) \end{array} \right.
		\;\mathrm{and}\;
		\left\{ \begin{array}{l} g''+bg=\lambda g \\ g'(0)=g'(1)=0 \end{array} \right.
	\end{equation*}
	with $f,g\in H^2(0,1)$. 
\\[2mm]
	If $g=0$ then $f''+a f = \lambda f$ and $f(1)=f'(0)=0$. Hence $\lambda=\lambda_{1,n}$ and $f = \phi_{1,n}^1$ for some arbitrary integer $n \geq 0$. Now if $g \neq 0$ we have $\lambda = \lambda_{2,m}$ and $g = \phi_{2,m}^2$ for some arbitrary integer $m \geq 0$. Hence, we have to solve $f''+af=\lambda_{2,m} f$ with $f(1)=0$ and $f'(0)=s\phi_{2,m}^2(0)=s\mu_m$. As a first case, let us assume that $\lambda_{2,m} = a$, i.e., $b-a = m^2 \pi^2$, yielding $f''=0$ with $f(1)=0$ and $f'(0)=s\mu_m$. A direct integration gives $f(x) = s \mu_m (-1+x)$. As a second case, let us assume that $\lambda_{2,m} \neq a$, i.e., $b-a \neq m^2 \pi^2$. Let $r_m$ be one of the two distinct square roots of $\lambda_{2,m}-a=b-a-m^2\pi^2$. Then $f(x) = \alpha e^{r_m x} + \beta e^{-r_m x}$ for some $\alpha,\beta\in\mathbb{C}$. The two conditions $f(1)=0$ and $f'(0)=s\mu_m$ imply
	\begin{equation}\label{eq: solve eigenvalues}
	\begin{bmatrix}
	e^{r_m} & e^{-r_m} \\ r_m & -r_m
	\end{bmatrix}
	\begin{bmatrix}
	\alpha \\ \beta
	\end{bmatrix}
	=
	\begin{bmatrix}
	0 \\ s\mu_m
	\end{bmatrix} .
	\end{equation}
	The latter square matrix is invertible if and only if its determinant $-2 r_m \cosh r_m \neq 0$. This holds if and only if $m \notin \Delta_2(a,b)$. In that case $\alpha = \frac{s\mu_me^{-r_m}}{2 r_m \cosh r_m }$ and $\beta = - \frac{s\mu_me^{r_m}}{2 r_m \cosh r_m }$, hence $f(x) = - \frac{s\mu_m}{r_m \cosh r_m} \sinh(r_m(1-x))$. 
\\[2mm]
	It remains to discuss the case where $m \in \Delta_2(a,b)$, that is: there exists an integer $k \geq 0$ such that $(k,m) \in \Delta(a,b)$, implying $\lambda = \lambda_{1,k} = \lambda_{2,m}$. In that case, it can be seen that there does not exist any solution of \eqref{eq: solve eigenvalues}. Hence, instead of searching an eigenvector, we search a generalized eigenvector for $\lambda$. Since $\phi_{1,k}$ is an eigenvector for $\lambda$, we search a solution $(f,g) \in D(\mathcal{A})$ of $\mathcal{A}(f,g) = \lambda (f,g) + \nu \phi_{1,k}$ for some $\nu\neq 0$ to be determined. With $g=\phi_{2,m}^2$, $f \in H^2(0,1)$ is solution of
	\begin{equation*}
		\left\{ \begin{array}{l} f''+af=\lambda f + \nu \phi_{1,k}^1 \\ f(1)=0 , \; f'(0) = s \mu_m \end{array} \right.
	\end{equation*}
	for some $\nu\neq 0$ to be determined. Noting that $\lambda-a=-(k+1/2)^2\pi^2$, the solution of $f''+af=\lambda f + \nu \phi_{1,k}^1$ takes the form
	\begin{equation*}
		f(x)  = \left( \gamma_1 + \frac{1}{2r_k} \int_0^x e^{-r_k s} \nu \phi_{1,k}^1(s) \,\mathrm{ds} \right) e^{r_k x} + \left( \gamma_2 - \frac{1}{2r_k} \int_0^x e^{r_k s} \nu \phi_{1,k}^1(s) \,\mathrm{ds} \right) e^{-r_k x}
	\end{equation*}		
	with $r_k = i(k+1/2)\pi$ for some $\gamma_1,\gamma_2,\nu\in\mathbb{R}$ to be determined so that $f(1)=0$ and $f'(0)=s\mu_m$. Evaluating the two latter identities shows that they hold if and only if $(\gamma_1-\gamma_2)r_k = s \mu_m = - \frac{\sqrt{2}}{2} \nu$. This is satisfied by setting $\nu = - \sqrt{2} s \mu_m$, $\gamma_1 = \gamma_2 + s\mu_m/r_k$, and $\gamma_2 = \frac{\sqrt{2}}{2} c_m$ for arbitrary $c_m \in\mathbb{R}$. Setting $\phi_{2,m}^1 = \mathrm{Re} f$, we infer that
	\begin{multline}
		 \phi_{2,m}^1(x) =  \frac{s \mu_m}{(k+1/2)\pi}\sin((k+1/2)\pi x)  - \frac{\sqrt{2}s\mu_m}{(k+1/2)\pi} \int_0^x \sin((k+1/2)\pi(x-s)) \phi_{1,k}^1(s) \,\mathrm{d}s \\ + c_{m} \sqrt{2} \cos((k+1/2)\pi x) \label{eq: eigenelements - constant cm}
	\end{multline}
	for an arbitrary $c_{m} \in \mathbb{R}$. A computation shows that $\int_0^x \sin((k+1/2)\pi(x-s)) \phi_{1,k}^1(s) \,\mathrm{d}s = \frac{\sqrt{2}}{2} x \sin((k+1/2)\pi x)$, yielding the claimed result by setting\footnote{The coefficient $c_m$ introduces a translation of the generalized eigenvector $\phi_{2,m}$ in the direction of the eigenvector $\phi_{1,k}$ associated to the same eigenvalue $\lambda$. Hence $c_m \in \mathbb{R}$ can be fixed arbitrarily. Note however that this choice fixes the corresponding constant that will appear in the computation of the dual Riesz basis provided by Lemma~\ref{eq: dual Riesz basis}.} $c_m = 0$.
\\[2mm]
	Let us finally prove that $\mathrm{Card}(\Delta(a,b)) < \infty$. If $(n,m) \in \Delta(a,b)$ then $\vert b - a \vert / \pi^2 = \vert m^2 - (n+1/2)^2 \vert = \vert m-n-1/2 \vert (m + n + 1/2) \geq (m + n + 1/2)/2 \geq \max(n,m)/2$, hence $\Delta(a,b) \subset \{ 0 , 1 , \ldots , \lfloor 2\vert b - a \vert / \pi^2 \rfloor \}^2$. Therefore $\Phi$ contains a finite number of generalized eigenvectors that are not eigenvectors of $\mathcal{A}$. \qed
\end{proof}

\begin{remark}
	For a given $a\in\mathbb{R}$, the set of $b \in \mathbb{R}$ such that $\Delta(a,b) \neq \emptyset$ is an infinite countable set consisting of isolated values. More precisely, if $b \neq b'$ are such that $\Delta(a,b) \neq \emptyset$ and $\Delta(a,b') \neq \emptyset$ then, based on \eqref{eq: def Delta(a,b)}, $\frac{4}{\pi^2}(b-a)$ and $\frac{4}{\pi^2}(b'-a)$ are two distinct integers, hence $\vert b - b' \vert \geq \pi^2/4$. The same conclusion holds when switching the roles of $a$ and $b$.
\end{remark}

\begin{example}
	Assume that $b-a = 35\pi^2/4$. Then $\Delta(a,b) = \{ (0,3),(8,9) \}$, hence $\Phi$ contains two generalized eigenvectors that are not eigenvectors of $\mathcal{A}$, namely, $\phi_{2,3}$ and $\phi_{2,9}$, associated with the eigenvalues $\lambda_{1,0} = \lambda_{2,3} = a - \pi^2/4$ and $\lambda_{1,8} = \lambda_{2,9} = a - 289\pi^2/4$ respectively. 
\end{example}

\begin{lemma}\label{lem: Riesz basis}
	The set $\Phi$ forms a Riesz basis of $\mathcal{H}^0$. Moreover, $\mathcal{A}$ generates a $C_0$-semigroup.
\end{lemma}

\begin{proof}
	Setting $\tilde{\phi}_{2,m} = ( 0 , \phi_{2,m}^2)$, the set $\tilde{\Phi} = \{ \phi_{1,n} \,\mid\, n\geq 0 \} \cup \{ \tilde{\phi}_{2,m} \,\mid\, m \geq 0\}$ is a Hilbert basis of $\mathcal{H}^0$.
	Using \eqref{eq: phi_2_m^1 - 2}, it follows from $\Vert \phi_{2,m}^1 \Vert_{L^\infty} = \mathrm{O}(1/m)$ that
	\begin{equation*}
	\sum_{m \geq 0} \Vert \phi_{2,m} - \tilde{\phi}_{2,m} \Vert_{\mathcal{H}^0}^2
	= \sum_{m \geq 0} \Vert \phi_{2,m}^1 \Vert_{L^2}^2 < \infty .
	\end{equation*}
	 Since $\Phi$ is $\omega$-linearly independent, Bari's theorem (see \cite{gohberg1978introduction}) implies that $\Phi$ is a Riesz basis.
	Finally, the Riesz basis property combined with the point spectrum described by \eqref{eq: point spectrum} and the fact that $\mathcal{A}$ is closed imply that $\mathcal{A}$ generates a $C_0$-semigroup (see, e.g.,  \cite{curtain2012introduction}). \qed
\end{proof}

We define the Hilbert space 
$$
H_{(1)}^1(0,1) = \{ f\in H^1(0,L) \,\mid\, f(1)=0 \}
$$
endowed with the inner product $\langle f_1 , f_2 \rangle = \int_0^1 f_1' f_2' \,\mathrm{d}x$. The previous lemma is instrumental for studying the solutions in $L^2$-norm. For the study of the solutions in $H^1$-norm, the following result is key. We define the Hilbert space
\begin{subequations}\label{eq: space H1}
\begin{equation}
	\mathcal{H}^1 = H_{(1)}^1(0,1) \times H^1(0,1) 
\end{equation}
endowed with the inner product
\begin{equation}
\langle (f_1,g_1) , (f_2,g_2) \rangle 
= \int_0^1 (f_1' f_2' + g_1 g_2 + g_1' g_2' ) \,\mathrm{d}x 
\end{equation}
\end{subequations}
and associated norm $\Vert \cdot \Vert_{\mathcal{H}_1}$.

\begin{lemma}\label{lem: Riesz basis H1}
	The set $\Phi^1 = \{ \frac{1}{(n+1/2)\pi}\phi_{1,n} \,\mid\, n\geq 0 \} \cup \{ \frac{1}{\sqrt{1+m^2\pi^2}}\phi_{2,m} \,\mid\, m \geq 0\}$ is a Riesz basis of $\mathcal{H}^1$.
\end{lemma}

\begin{proof}
The family $\{ \frac{1}{(n+1/2)\pi} \sqrt{2} \cos\left( (n+1/2) \pi x \right) \,\mid\, n\in\mathbb{N} \}$ is a Hilbert basis of $H^1_{(1)}(0,1)$ and the family $\{ \frac{1}{\sqrt{1+m^2\pi^2}} \mu_m \cos\left( m \pi x \right) \,\mid\, m\in\mathbb{N} \}$ is a Hilbert basis of $H^1(0,1)$. Therefore, $\tilde{\Phi}^1 = \{ \frac{1}{(n+1/2)\pi}\phi_{1,n} \,\mid\, n\geq 0 \} \cup \{ \frac{1}{\sqrt{1+m^2\pi^2}} (0,\phi_{2,m}^2) \,\mid\, m \geq 0\}$ is a Hilbert basis of $\mathcal{H}^1$. To invoke Bari's theorem, since $\Phi^1$ is $\omega$-linearly independent, it remains to show that $\sum_{m\geq 0} \left\Vert \frac{1}{\sqrt{1+m^2\pi^2}} \phi^1_{2,m}\right\Vert_{H_{(1)}^1}^2 = \sum_{m\geq 0} \frac{1}{1+m^2\pi^2} \Vert \phi^1_{2,m} \Vert_{H_{(1)}^1}^2$ is finite. For $m$ sufficiently large, we know from \eqref{eq: phi_2_m^1 - 2} that $(\phi_{2,m}^1)'(x) = \frac{s \mu_m}{\cosh(r_m)}\cosh(r_m(1-x))$ where $r_m$ is one of the two square roots of $\lambda_{2,m}-a=b-a-m^2\pi^2$. Hence $\Vert (\phi^1_{2,m})' \Vert_{L^\infty} = \mathrm{O}(1)$, which completes the proof. \qed
\end{proof}

\subsection{Dual Riesz basis}
In order to expand the solutions in the Riesz basis $\Phi$, we need to characterize its dual Riesz basis $\Psi$. 
An easy computation shows that the adjoint operator of $\mathcal{A}$ is given by
	\begin{subequations}\label{eq: adjoint operator}
	\begin{equation}
	\mathcal{A}^*(f,g) = (f''+af,g''+bg)
	\end{equation}
	on the domain
	\begin{equation}
		D(\mathcal{A}^*) = \big\{ (f,g) \in H^2(0,1) \times H^2(0,1) \,\mid\, 	
		 f(1)=f'(0)=g'(1)=0 ,\, g'(0) = s f(0) \big\} . 
	\end{equation}
	\end{subequations}

\begin{lemma}\label{eq: dual Riesz basis}
	The dual Riesz basis of $\Phi$ is the set $\Psi = \{ \psi_{1,n} \,\mid\, n\geq 0 \} \cup \{ \psi_{2,m} \,\mid\, m \geq 0\}$ of generalized eigenvectors of $\mathcal{A}^*$, given by $\psi_{2,m} = (\psi_{2,m}^1,\psi_{2,m}^2)$ for $m \geq 0$, with $\psi_{2,m}^1 = 0$ and 
	\begin{equation*}
		\psi_{2,m}^2(x) = 
		\left\{
		\begin{array}{ll}
			1  & \quad \textrm{if}\ \ m = 0, \\
			\sqrt{2} \cos\left( m \pi x \right)  & \quad \textrm{if}\ \ m \geq 1,
		\end{array}
		\right.
	\end{equation*}
	and $\psi_{1,n} = (\psi_{1,n}^1,\psi_{1,n}^2)$ for $n \geq 0$ with
	\begin{equation*}
		\psi_{1,n}^1(x) = \sqrt{2} \cos\left( (n+1/2) \pi x \right)
	\end{equation*}
	and $\psi_{1,n}^2$ defined as follows:
	\begin{subequations}
	\begin{itemize}
	\item For $n \notin \Delta_1(a,b)$,
	\begin{equation}\label{eq: psi_1_n^2 - 1}
		\psi_{1,n}^2(x) = -\frac{\sqrt{2} s}{r_n^* \sinh(r_n^*)}\cosh(r_n^*(1-x)) 
	\end{equation}
	where $r_n^*$ is one of the two square roots of $\lambda_{1,n}-b$. Note that this function is real valued.
	\item For $n \in \Delta_1(a,b)$, we denote by $k\in\mathbb{N}$ the integer such that $(n,k)\in\Delta(a,b)$.
	\begin{itemize}
	\item[--] For $k=0$,
	\begin{equation}\label{eq: psi_1_n^2 - 2}
		\psi_{1,n}^2(x) = c_0 + \sqrt{2} s x - \frac{\sqrt{2}}{2} s x^2
	\end{equation}
	with $c_0 = - s \sqrt{2} \left( \frac{1}{3} + \frac{1}{(2n+1)^2\pi^2} \right)$.
	\item[--] For $k \geq 1$,
	\begin{equation}\label{eq: psi_1_n^2 - 3}
		\psi_{1,n}^2(x) =  \frac{\sqrt{2}s}{k\pi} (1-x) \sin(k\pi x) + c_n \sqrt{2} \cos(k\pi x) 
	\end{equation}
	with $c_n = - \frac{2s}{\pi^2} \left( \frac{1}{(2n+1)^2} + \frac{1}{4k^2} \right)$.
	\end{itemize}
	\end{itemize}
	\end{subequations}
	Moreover:
	\begin{enumerate}
		\item If $n \notin \Delta_1(a,b)$ then $\lambda_{1,n}$ is an eigenvalue of $\mathcal{A}^*$ with associated eigenvector $\psi_{1,n}$:
		\begin{equation}\label{eq: A* psi_1_n}
			\mathcal{A}^* \psi_{1,n} = \lambda_{1,n} \psi_{1,n} .
		\end{equation}	
		\item $\lambda_{2,m}$ is an eigenvalue of $\mathcal{A}^*$ with associated eigenvector $\psi_{2,m}$:
		\begin{equation}\label{eq: A* psi_2_m}
			\mathcal{A}^* \psi_{2,m} = \lambda_{2,m} \psi_{2,m}.
		\end{equation}	
		\item $(n,m)\in\Delta(a,b)$ if and only if $\lambda = \lambda_{1,n} = \lambda_{2,m}$. In that case $\lambda$ is an eigenvalue of $\mathcal{A}$ of geometric multiplicity $1$ but of algebraic mutiplicity $2$, and
		\begin{subequations}\label{eq: A* psi_1_n psi_2_m mult alg 2}
		\begin{align}
			\mathcal{A}^* \psi_{1,n} & = \lambda \psi_{1,n} - \sqrt{2} s \mu_m \psi_{2,m} \\
			\mathcal{A}^* \psi_{2,m} & = \lambda \psi_{2,m}
		\end{align}	
		\end{subequations}
		where $\mu_m = 1$ if $m=0$ and $\mu_m = \sqrt{2}$ if $m \geq 1$.		
		\item $\Psi$ contains the finite number $\mathrm{Card}(\Delta(a,b))$ of generalized eigenvectors that are not eigenvectors of $\mathcal{A}^*$.
	\end{enumerate} 
\end{lemma}

\begin{proof}
Similarly as in the proof of Lemma~\ref{lem: eigenelements A}, we search solutions $(f,g)\in D(\mathcal{A}^*)$ of $\mathcal{A}^*(f,g)=\lambda(f,g)$. We only give the details for the case $n\in\Delta_1(a,b)$. There exists an integer $k \geq 0$ such that $(n,k)\in\Delta(a,b)$, hence $\lambda = \lambda_{1,n} = \lambda_{2,k}$. The geometric multiplicity of $\lambda$ is $1$. Hence, knowing that $\psi_{2,k}$ is an eigenvector for $\lambda$, we search a generalized eigenvector, that is a solution $(f,g)\in D(\mathcal{A}^*)$ of $\mathcal{A}^*(f,g)=\lambda(f,g)+\nu\psi_{2,k}$ for some $\nu\neq0$ to be determined. We detail the case $k \geq 1$, the case $k=0$ being similar. Proceeding as in the proof of Lemma~\ref{lem: eigenelements A}, we infer that
	\begin{equation*}
		 \psi_{1,n}^2(x) =  \frac{\sqrt{2}s}{k\pi}\sin(k\pi x) + c_n \sqrt{2} \cos(k\pi x) 
		 - \frac{2s}{k\pi} \int_0^x \sin(k\pi(x-s)) \psi_{2,k}^2(s) \,\mathrm{d}s 
	\end{equation*}
	for an arbitrary $c_n\in\mathbb{R}$. A direct integration shows that $\int_0^x \sin(k\pi(x-s)) \psi_{2,k}^2(s) \,\mathrm{d}s = \frac{\sqrt{2}}{2} x \sin(k \pi x)$, yielding
	\begin{equation*}
		\psi_{1,n}^2(x) =  \frac{\sqrt{2}s}{k\pi} (1-x) \sin(k\pi x) + c_n \sqrt{2} \cos(k\pi x) .
	\end{equation*}
	It remains to compute $c_n$. Since we want to ensure that $\Psi$ is the dual Riesz basis of $\Phi$, we must select $c_n$ so that $\langle \phi_{2,k} , \psi_{1,n} \rangle = 0$. Easy computations give the claimed results. \qed
\end{proof}

\section{Spectral reduction and controllability of the modes}\label{sec: spectral reduction}

\subsection{Homogeneous representation}
The state at time $t$ is $\mathcal{X}(t) = (y(t,\cdot),z(t,\cdot))$.
Let us make a change of variable.
We define $\varphi(x) = x-x^2/2$, so that $\varphi(0) = \varphi'(1) = 0$ and $\varphi'(0) = 1$, and we set
\begin{equation}\label{eq: change of variable}
	\tilde{z}(t,x) = z(t,x) - \varphi(x) u(t) .
\end{equation}
Assuming that $u$ is derivable (which will be structurally ensured by our choice of control strategy later on), we infer from \eqref{eq: system} that
	\begin{align*}
		& \partial_t y = \partial_{xx} y + a y ,\\
		& \partial_t \tilde{z} = \partial_{xx} \tilde{z} + b \tilde{z} + \alpha u + \beta \dot{u}, \\
		& \partial_x y(t,0) = s \tilde{z}(t,0) ,\ \ y(t,1) = \partial_x \tilde{z}(t,1) = \partial_x \tilde{z}(t,0) = 0 ,
	\end{align*}
with $\alpha(x) = \varphi''(x) + b \varphi(x)$ and $\beta(x) = - \varphi(x)$. Setting $\tilde{\mathcal{X}}(t) = (y(t,\cdot),\tilde{z}(t,\cdot))$ and $v = \dot{u}$, we infer that
\begin{align}
	\frac{\mathrm{d}\tilde{\mathcal{X}}}{dt} &= \mathcal{A} \tilde{\mathcal{X}} + (0,\alpha) u + (0,\beta) v , \label{eq: abstract system homogeneous} \\
	\dot u &= v .
\end{align}
The new state is thus the pair $(\tilde{\mathcal{X}},u)$: not only we have made a change of variable, but we have also augmented the state by considering $u$ as a state and $v=\dot u$ as the new control.
Note that
\begin{equation}\label{eq: change of variable abstract}
\tilde{\mathcal{X}} = \mathcal{X} + (0,\beta) u .
\end{equation}

\subsection{Dynamics of the modes}

In order to expand the solutions in the Riesz basis $\Phi$, we define the coefficients $\tilde{x}_{1,n} = \langle \tilde{\mathcal{X}} , \psi_{1,n} \rangle$, $\tilde{x}_{2,m} = \langle \tilde{\mathcal{X}} , \psi_{2,m} \rangle$, $x_{1,n} = \langle \mathcal{X} , \psi_{1,n} \rangle$, $x_{2,m} = \langle \mathcal{X} , \psi_{2,m} \rangle$, $\alpha_{1,n} = \langle (0,\alpha) , \psi_{1,n} \rangle , \quad \alpha_{2,m} = \langle (0,\alpha) , \psi_{2,m} \rangle$, $\beta_{1,n} = \langle (0,\beta) , \psi_{1,n} \rangle$, and $\beta_{2,m} = \langle (0,\beta) , \psi_{2,m} \rangle$. Then, using Lemma~\ref{lem: Riesz basis}, we have 
\begin{equation}
\tilde{\mathcal{X}}(t)  = (y(t,\cdot),\tilde{z}(t,\cdot))  = \sum_{n \geq 0} \tilde{x}_{1,n}(t) \phi_{1,n} + \sum_{m \geq 0} \tilde{x}_{2,m}(t) \phi_{2,m} . \label{eq: series expansion of the state}
\end{equation}
Expanding \eqref{eq: change of variable abstract} in the Riesz basis $\Psi$ gives
\begin{subequations}\label{eq: projection change of basis formula}
\begin{align}
	\tilde{x}_{1,n} & = x_{1,n} + \beta_{1,n} u , \quad && n \geq 0 , \label{eq: projection change of basis formula - 1} \\
	\tilde{x}_{2,m} & = x_{2,m} + \beta_{2,m} u , \quad && m \geq 0 . \label{eq: projection change of basis formula - 2}
\end{align}
\end{subequations}
\begin{itemize}
\item For $n \notin \Delta_1(a,b)$, using \eqref{eq: A* psi_1_n}, the projection of \eqref{eq: abstract system homogeneous} onto $\phi_{1,n}$ gives
\begin{equation}\label{eq: dynamics tilde_x1n}
	\dot{\tilde{x}}_{1,n} = \lambda_{1,n} \tilde{x}_{1,n} + \alpha_{1,n} u + \beta_{1,n} v
\end{equation}
that is, using the change of variable formula \eqref{eq: projection change of basis formula - 1} and recalling that $v = \dot{u}$, 
\begin{equation}\label{eq: dynamics x1n}
	\dot{x}_{1,n} = \lambda_{1,n} x_{1,n} + \nu_{1,n} u 
\end{equation}
where $\nu_{1,n} = \alpha_{1,n} + \lambda_{1,n} \beta_{1,n}$.
\item For $m \geq 0$, using \eqref{eq: A* psi_2_m}, the projection of \eqref{eq: abstract system homogeneous} onto $\phi_{2,m}$ gives
\begin{equation}\label{eq: dynamics tilde_x2m}
	\dot{\tilde{x}}_{2,m} = \lambda_{2,m} \tilde{x}_{2,m} + \alpha_{2,m} u + \beta_{2,m} v .
\end{equation}
From the change of variable formula \eqref{eq: projection change of basis formula - 2}, we have 
\begin{equation}\label{eq: dynamics x2m}
	\dot{x}_{2,m} = \lambda_{2,m} x_{2,m} + \nu_{2,m} u 
\end{equation}
where $\nu_{2,m} = \alpha_{2,m} + \lambda_{2,m} \beta_{2,m}$.
\item For $n\in\Delta_1(a,b)$, let $m \geq 0$ be such that $(n,m)\in\Delta(a,b)$, meaning that $\lambda = \lambda_{1,n} = \lambda_{2,m}$. In that case, using \eqref{eq: A* psi_1_n psi_2_m mult alg 2}, the projection of \eqref{eq: abstract system homogeneous} onto $\phi_{1,n}$ and $\phi_{2,m}$ gives
\begin{equation}\label{eq: dynamics tilde_x1n tilde_x2m mult 2}
	\begin{bmatrix} \dot{\tilde{x}}_{1,n} \\ \dot{\tilde{x}}_{2,m} \end{bmatrix}
	= \underbrace{\begin{bmatrix} \lambda & -\sqrt{2} s \mu_m \\ 0 & \lambda \end{bmatrix}}_{= M_{n}} \begin{bmatrix} \tilde{x}_{1,n} \\ \tilde{x}_{2,m} \end{bmatrix} + \begin{bmatrix} \alpha_{1,n} \\ \alpha_{2,m} \end{bmatrix} u + \begin{bmatrix} \beta_{1,n} \\ \beta_{2,m} \end{bmatrix} v
\end{equation}
hence, based on \eqref{eq: projection change of basis formula}, 
\begin{equation}\label{eq: dynamics x1n x2m mult 2}
	\begin{bmatrix} \dot{x}_{1,n} \\ \dot{x}_{2,m} \end{bmatrix}
	= M_n \begin{bmatrix} x_{1,n} \\ x_{2,m} \end{bmatrix} + \begin{bmatrix} \nu_{1,n} \\ \nu_{2,m} \end{bmatrix} u 
\end{equation}
where $\nu_{1,n} = \alpha_{1,n} + \lambda_{1,n} \beta_{1,n} - \sqrt{2}s\mu_m \beta_{2,m}$ and $\nu_{2,m} = \alpha_{2,m} + \lambda_{2,m} \beta_{2,m}$.
\end{itemize}

\subsection{Controllability properties}

In order to characterize the controllabity of each mode of the PDE cascade \eqref{eq: system}, we first evaluate $\nu_{1,n}$ and $\nu_{2,m}$.

\begin{lemma}\label{lem: mu_i_k}
$\nu_{i,k} = - \psi_{i,k}^2(0)$ for all $k \in\mathbb{N}$ and $i\in\{1,2\}$.
\end{lemma}

\begin{proof}
For $k \geq 0$, in the case where either (i) $i = 2$; or (ii) $i = 1$ with $k \notin \Delta_1(a,b)$, we evaluate the terms $\nu_{i,k} = \alpha_{i,k} + \lambda_{i,k} \beta_{i,k}$ appearing in the dynamics \eqref{eq: dynamics x1n} and \eqref{eq: dynamics x2m}, as follows:
$$
	\nu_{i,k}
	 = \langle (0,\alpha) , \psi_{i,k} \rangle + \lambda_{i,k} \langle (0,\beta) , \psi_{i,k} \rangle 
	 = \langle \varphi''+b\varphi , \psi_{i,k}^2 \rangle - \lambda_{i,k} \langle \varphi , \psi_{i,k}^2 \rangle .
$$
Using integration by parts we have
\begin{multline*}
	 \langle \varphi'' , \psi_{i,k}^2 \rangle = \int_0^1 \varphi''(x) \psi_{i,k}^2(x) \,\mathrm{d}x 
	 = \left[ \varphi'(x) \psi_{i,k}^2(x) - \varphi(x) (\psi_{i,k}^2)'(x) \right]_0^1  + \int_0^1 \varphi(x) (\psi_{i,k}^2)''(x) \,\mathrm{d}x \\
	 = - \psi_{i,k}^2(0) + \langle \varphi , (\psi_{i,k}^2)'' \rangle
\end{multline*}
where we have used the conditions $\varphi(0) = \varphi'(1) = 0$, $\varphi'(0) = 1$, and $(\psi_{i,k}^2)'(1)=0$. This shows that
\begin{align*}
	\nu_{i,k}
	& = - \psi_{i,k}^2(0) + \langle \varphi , (\psi_{i,k}^2)'' + b \psi_{i,k}^2 \rangle - \lambda_{i,k} \langle \varphi , \psi_{i,k}^2 \rangle \\
	& = - \psi_{i,k}^2(0) + \langle (0,\varphi) , \mathcal{A}^* \psi_{i,k} \rangle - \lambda_{i,k} \langle (0,\varphi) , \psi_{i,k} \rangle
\end{align*}
which gives $\nu_{i,k} = - \psi_{i,k}^2(0)$ because $\mathcal{A}^* \psi_{i,k} = \lambda_{i,k} \psi_{i,k}$.
\\[2mm]
Consider now the case $i = 1$ and $k\in\Delta_1(a,b)$. Let $l\geq0$ be the integer such that $(k,l) \in \Delta(a,b)$. Then $\nu_{1,k} = \alpha_{1,k} + \lambda_{1,k} \beta_{1,k} - \sqrt{2}s\mu_l \beta_{2,l}$ and the previous approach shows that 
\begin{equation*}
	\nu_{1,k}
	= - \psi_{1,k}^2(0) + \langle (0,\varphi) , \mathcal{A}^* \psi_{1,k} \rangle 
	- \lambda_{1,k} \langle (0,\varphi) , \psi_{1,k} \rangle - \sqrt{2}s\mu_l \beta_{2,l}
\end{equation*}
which gives $\nu_{1,k} = - \psi_{1,k}^2(0)$ because $\mathcal{A}^* \psi_{1,k} = \lambda_{1,k} \psi_{1,k} - \sqrt{2} s \mu_l \psi_{2,l}$. \qed
\end{proof}

We are now in a position to characterize the controllability property for each individual mode of the PDE cascade \eqref{eq: system}. This analysis is extended in Appendix~\ref{appendix} to the study of the exact observability properties of the full PDE system \eqref{eq: system} in suitable Hilbert spaces.
\begin{itemize}
\item For $n \notin \Delta_1(a,b)$, the mode $\lambda_{1,n}$ has a one-dimensional dynamics given by \eqref{eq: dynamics x1n}. It is controllable if and only if $\nu_{1,n} = - \psi_{1,n}^2(0) = \frac{\sqrt{2}s}{r_n^*} \coth(r_n^*) \neq 0$, that is $\cosh(r_n^*) \neq 0$. Recalling that $r_n^*$ is one of the two square roots of $\lambda_{1,n}-b=a-b-(n+1/2)^2\pi^2$, the latter condition holds if and only if $n \notin \Theta(a,b)$ with
\begin{equation}
		\Theta(a,b)  = \{ n\in\mathbb{N} \,\mid\, \exists k\in\mathbb{N}  \;\mathrm{s.t.} \  
		 4(b-a) = ( (2k+1)^2 - (2n+1)^2 ) \pi^2 \}. \label{eq: def Theta(a,b)}
\end{equation}
\item For $m \notin \Delta_2(a,b)$, the mode $\lambda_{2,m}$ has a one-dimensional dynamics given by \eqref{eq: dynamics x2m}, which is controllable because $\nu_{2,m} = - \psi_{2,m}^2(0) = - \mu_m \neq 0$.
\item For $(n,m)\in\Delta(a,b)$, we have $\lambda = \lambda_{1,n} = \lambda_{2,m}$ with a two-dimensional dynamics given by \eqref{eq: dynamics x1n x2m mult 2}. This dynamics is controllable because $s\mu_m \neq 0$ and $\nu_{2,m} = - \psi_{2,m}^2(0) = - \mu_m \neq 0$.
\end{itemize}

\begin{remark}\label{rmk: discussion Theta(a,b)}
The loss of controllability for at least one mode of the system \eqref{eq: system} occurs if and only if the reaction terms $a,b$ are selected such that $\Theta(a,b) \neq \emptyset$.

This controllability result for the modes indicates that the PDE cascade \eqref{eq: system} cannot be exponentially stabilized as soon as there exists an integer $n \geq 0$ for which $n \in \Theta(a,b)$ and $\lambda_{1,n} \geq 0$. A first example of such a scenario is when $a=b$, which gives $\Theta(a,b) = \mathbb{N}$. In this situation, none of the modes $\lambda_{1,n}$ associated with the second PDE \eqref{eq: system - y} of the cascade is controllable. In this case, the system cannot be exponentially stabilized as soon as $\lambda_{1,0} \geq 0$, that is when $a=b \geq \pi^2/4$. A second example for $a \neq b$ is $a = \pi^2$ and $b = 3 \pi^2$ for which the unstable mode $\lambda_{1,0} = 3\pi^2/4 > 0$ is not controllable since $\lambda_{1,0} - b = - 9 \pi^2/4$, hence $r_0^* = 3i\pi/4$ and $\nu_1(0) = - \psi_{1,n}^2(0) = \frac{\sqrt{2}s}{r_0^*} \coth(r_0^*) = 0$. 
\end{remark}

\begin{remark}
Note that, as soon as $a \neq b$, we have $\mathrm{Card}(\Theta(a,b)) < \infty$ with $n \leq \vert b-a \vert / \pi^2$ for all $n \in \Theta(a,b)$. Indeed, for any $n \in \Theta(a,b)$ there exists $k\in\mathbb{N}$ such that $4(b-a)/\pi^2 = (2k+1)^2 - (2n+1)^2 = 4 (k-n)(k+n+1)$. Since $a \neq b$ we must have $k \neq n$ hence $\vert b-a \vert / \pi^2 \geq k+n+1 \geq n$.  
\end{remark}

\begin{remark}
	For a given $a\in\mathbb{R}$, the set of $b \in \mathbb{R}$ such that $\Theta(a,b) \neq \emptyset$ is an infinite countable set consisting of isolated values. More precisely, if $b \neq b'$ are such that $\Theta(a,b) \neq \emptyset$ and $\Theta(a,b') \neq \emptyset$ then, based on \eqref{eq: def Theta(a,b)}, $\frac{4}{\pi^2}(b-a)$ and $\frac{4}{\pi^2}(b'-a)$ are two distinct integers, hence $\vert b - b' \vert \geq \pi^2/4$. The same conclusion holds when switching the roles of $a$ and $b$.
\end{remark}

\begin{remark}\label{rem: system 2}
The spectral reduction done in this section for the PDE cascade \eqref{eq: system} also applies in a straightforward way to the PDE cascade \eqref{eq: system 2} by defining in the change of variable formula \eqref{eq: change of variable} the function $\varphi(x)=x^2/2$ that was selected so that $\varphi(0)=\varphi'(0)=0$ and $\varphi'(1)=1$. By doing so, the only change in the above developments is that $\nu_{i,k} = \psi_{i,k}^2(1)$. Noting that $\psi_{2,m}^2(1) = (-1)^m \mu_m \neq 0$ and, in the case $n\notin\Delta_1(a,b)$, $\psi_{1,n}^2(1) = -\frac{\sqrt{2} s}{r_n^* \sinh(r_n^*)} \neq 0$, then it appears that each mode $\lambda_{i,k}$ of the PDE cascade \eqref{eq: system 2} is always controllable. This result highlights the difference between the two studied PDE cascades. While for the PDE cascade \eqref{eq: system} both control input $u$ and coupling between the two heat equations are collocated at the same boundary ($x=0$) for the $z$-equation, the control $u$ and the coupling occur at different boundaries in the PDE cascade \eqref{eq: system 2}. Hence, it appears that a non-collocated setting is more favorable to promote controllability properties of the plant. 
\end{remark}

\section{State-feedback}\label{sec: state-feedback}
In this section, we design an explicit state-feedback control strategy for the PDE cascade \eqref{eq: system}, based on the spectral analysis performed previously.
\\[2mm]
We first sort the elements of $\Delta(a,b)$. When $(n_i,m_i)\in\Delta(a,b)$, we have: (i) $n_1 = n_2$ if and only if $m_1 = m_2$; (ii) $n_1 < n_2$ if and only if $m_1 < m_2$. So, with $N_\Delta = \mathrm{Card}(\Delta(a,b))$, we can define two increasing functions $\theta_{1},\theta_{2} : \{ 0 , \ldots , N_\Delta-1 \} \rightarrow \mathbb{N}$ such that $\Delta(a,b) = \{ ( \theta_{1}(i) , \theta_{2}(i)) \,\mid\, 0 \leq i \leq N_\Delta-1 ) \}$. This implies that $\Delta_1(a,b) = \{ \theta_{1}(i) \,\mid\, 0 \leq i \leq N_\Delta-1 ) \}$ and $\Delta_2(a,b) = \{ \theta_{2}(i) \,\mid\, 0 \leq i \leq N_\Delta-1 ) \}$.
\\[2mm]
Let $\delta > 0$ be a targeted exponential decay rate for the closed-loop system. Since $\lambda_{1,n}$ and $\lambda_{2,m}$ tend $-\infty$ as $n,m \rightarrow +\infty$, we fix integers $N_{0},M_{0} \geq 0$ such that $\lambda_{1,N_{0}+1} < -\delta$ and $\lambda_{2,M_{0}+1} < -\delta$. In order to avoid cumbersome notations, and to keep the presentation as concise as possible, we further choose the integers $N_0,M_0$ such that $N_0 \geq \theta_1(N_\Delta) = \max\Delta_1(a,b)$ and $M_0 \geq \theta_2(N_\Delta) = \max\Delta_2(a,b)$. This implies that all modes of the PDE cascade \eqref{eq: system} that are not algebraically simple are captured among the $0 \leq n \leq N_0$ and $0 \leq m \leq M_0$ first modes. Even unnecessary, this choice  considerably simplifies the writing of the dynamics, the definition of the Lyapunov functions, and the subsequent stability analysis.
\\[2mm]
We now define $\varphi_1$ (resp., $\varphi_2$) as a permutation of $\{ 0 , \ldots , N_0 \}$ (resp., of $\{ 0 , \ldots , M_0 \}$) so that the indices corresponding to the modes that are not algebraically simple are listed first. More precisely, we define $\varphi_1(i) = \theta_1(i)$ for $i\in\{0,\ldots,N_\Delta-1\}$ while $\varphi_1$ is increasing on $\{ N_\Delta , \ldots , N_0 \}$. Similarly, $\varphi_2(i) = \theta_2(i)$ for $i\in\{0,\ldots,N_\Delta-1\}$ while $\varphi_2$ is increasing on $\{ N_\Delta , \ldots , M_0 \}$.
\\[2mm]
We are now in position to write the dynamics satisfied by the $0 \leq n \leq N_0$ and $0 \leq m \leq M_0$ first modes of the PDE cascade \eqref{eq: system}. Defining the vectors
\begin{align*}
	X_{1,1} & = \begin{bmatrix} \tilde{x}_{1,\varphi_1(0)} & \tilde{x}_{2,\varphi_2(0)} &  \ldots & \tilde{x}_{1,\varphi_1(N_\Delta-1)} & \tilde{x}_{2,\varphi_2(N_\Delta-1)} \end{bmatrix}^\top , \\
	X_{1,2} & = \begin{bmatrix} \tilde{x}_{1,\varphi_1(N_\Delta)} & \ldots & \tilde{x}_{1,\varphi_1(N_{0})} \end{bmatrix}^\top , \\
	X_{1,3} & = \begin{bmatrix} \tilde{x}_{2,\varphi_2(N_\Delta)} & \ldots & \tilde{x}_{2,\varphi_2(M_{0})} \end{bmatrix}^\top ,
\end{align*}
we infer from \eqref{eq: dynamics tilde_x1n}, \eqref{eq: dynamics tilde_x2m} and \eqref{eq: dynamics tilde_x1n tilde_x2m mult 2} that
\begin{equation*}
	\dot{X}_{1,i} = A_{1,i} x_{1,i} + B_{1,u,i} u + B_{1,v,i} v , \quad i\in\{1,2,3\},
\end{equation*}
where 
\begin{align*}
A_{1,1} & = \mathrm{diag}(M_{\varphi_1(0)},\ldots,M_{\varphi_1(N_\Delta-1)}) , \quad
A_{1,2}  = \mathrm{diag}(\lambda_{1,\varphi_1(N_\Delta)},\ldots,\lambda_{1,\varphi_1(N_{0})}) , \\
A_{1,3} & = \mathrm{diag}(\lambda_{2,\varphi_2(N_\Delta)},\ldots,\lambda_{2,\varphi_2(M_{0})}) , \\
B_{1,u,1} & = \begin{bmatrix} \alpha_{1,\varphi_1(0)} & \alpha_{2,\varphi_2(0)} & \ldots & \alpha_{1,\varphi_1(N_\Delta-1)} & \alpha_{2,\varphi_2(N_\Delta-1)} \end{bmatrix}^\top , \\
B_{1,u,2} & = \begin{bmatrix} \alpha_{1,\varphi_1(N_\Delta)} & \ldots & \alpha_{1,\varphi_1(N_{0})} \end{bmatrix}^\top , \quad
B_{1,u,3} = \begin{bmatrix} \alpha_{2,\varphi_2(N_\Delta)} & \ldots & \alpha_{2,\varphi_2(M_{0})} \end{bmatrix}^\top , \\
B_{1,v,1} & = \begin{bmatrix} \beta_{1,\varphi_1(0)} & \beta_{2,\varphi_2(0)} & \ldots & \beta_{1,\varphi_1(N_\Delta-1)} & \beta_{2,\varphi_2(N_\Delta-1)} \end{bmatrix}^\top , \\
B_{1,v,2} & = \begin{bmatrix} \beta_{1,\varphi_1(N_\Delta)} & \ldots & \beta_{1,\varphi_1(N_{0})} \end{bmatrix}^\top , \quad
B_{1,v,3}  = \begin{bmatrix} \beta_{2,\varphi_2(N_\Delta)} & \ldots & \beta_{2,\varphi_2(M_{0})} \end{bmatrix}^\top .
\end{align*}
Defining the vector $X_1 = \mathrm{col}( X_{1,1} , X_{1,2} , X_{1,3} )$, this implies that
\begin{equation}\label{eq: dynamics X1}
	\dot{X}_1 = A_1 X_1 + B_{1,u} u + B_{1,v} v 
\end{equation}
where $A_1 = \mathrm{diag}(A_{1,1},A_{1,2},A_{1,3})$,
\begin{equation*}
B_{1,u} = \begin{bmatrix} B_{1,u,1} \\ B_{1,u,2} \\ B_{1,u,3} \end{bmatrix} ,\quad
B_{1,v} = \begin{bmatrix} B_{1,v,1} \\ B_{1,v,2} \\ B_{1,v,3} \end{bmatrix} .
\end{equation*}
Finally, augmenting the vector $X_1$ as $X_{1,a} = \mathrm{col}( X_1 , u )$, we infer that
\begin{equation}\label{eq: augmented truncated model}
	\dot{X}_{1,a} = A_{1,a} X_{1,a} + B_{1,a} v
\end{equation}
with
\begin{equation*}
	A_{1,a} = \begin{bmatrix} A_1 & B_{1,u} \\ 0 & 0  \end{bmatrix} , \quad
	B_{1,a} = \begin{bmatrix} B_{1,v} \\ 1 \end{bmatrix} .
\end{equation*}
Denoting $\nu_{i,k} = \alpha_{i,k} + \lambda_{i,k} \beta_{i,k}$, based on the discussion in the previous section regarding the controllability of the modes, the application of the Hautus test shows that the pair $(A_{1,a},B_{1,a})$ is controllable if and only if $\nu_{1,n} \neq 0$ for all $n \in \{\varphi_1(N_\Delta),\ldots,\varphi_1(N_{0})\}$. Noting that $\Delta_1(a,b) \cap \Theta(a,b) = \emptyset$, this gives the following result. 

\begin{lemma}\label{lem: cont finite dim model}
The pair $(A_{1,a},B_{1,a})$ is controllable if and only if 
\begin{equation}\label{eq: controllability assumption}
\Theta(a,b) \cap \{0,\ldots,N_{0}\} = \emptyset
\end{equation}
where $\Theta(a,b)$ is defined by \eqref{eq: def Theta(a,b)}.
\end{lemma}

Under this condition, one can design a state-feedback of the form 
\begin{equation}\label{eq: state-feedback}
v = - K X_{1,a}
\end{equation}
(for instance, by usual pole shifting or by Riccati theory)
to stabilize the augmented plant \eqref{eq: augmented truncated model}, hence ensuring the stabilization of the $0 \leq n \leq N_0$ and $0 \leq m \leq M_0$ first modes of the PDE cascade \eqref{eq: system}. The next result shows that this state-feedback strategy actually stabilizes the full PDE cascade \eqref{eq: system} in $L^2$- and $H^1$-norms.

\begin{theorem}\label{thm: state-feedback}
	Let $a,b,s \in\mathbb{R}$ with $s \neq 0$. Let $\delta > 0$ be arbitrary. Let $N_{0},M_{0} \geq 0$ be such that $\lambda_{1,N_{0}+1} < -\delta$ and $\lambda_{2,M_{0}+1} < -\delta$ with $N_0 \geq \max\Delta_1(a,b)$ and $M_0 \geq \max\Delta_2(a,b)$. Assuming that the controllability assumption \eqref{eq: controllability assumption} holds true, let $K\in\mathbb{R}^{1 \times (N_{0}+M_{0}+3)}$ be a matrix such that $A_{1,a} - B_{1,a} K$ is Hurwitz with spectral abscissa less than $-\delta$. Then there exists $C > 0$ such that, for all initial conditions $y_0,z_0 \in L^2(0,1)$ and $u(0)=u_0\in\mathbb{R}$, the solutions of the PDE system \eqref{eq: system} in closed-loop with the feedback \eqref{eq: state-feedback} 
satisfy
	\begin{equation}\label{eq: state-feedback - exp estimate}
		\Vert ( y(t,\cdot) , z(t,\cdot) ) \Vert_{Y}^2 + u(t)^2  
		\leq C e^{-2\delta t} \left( \Vert ( y_0 , z_0 ) \Vert_{Y}^2 + u_0^2 \right) \qquad\forall t\geq 0 
	\end{equation}
	with $Y = \mathcal{H}^0$ defined by \eqref{eq: space H0}. If the initial conditions are such that $y_0,z_0\in H^2(0,1)$ with $y_0(1)=z_0'(1)=0$, $y_0'(0)=sz_0(0)$ and $z_0'(0)=u(0)$, then \eqref{eq: state-feedback - exp estimate} holds with $Y = \mathcal{H}^1$ defined by \eqref{eq: space H1}.
\end{theorem} 

\begin{proof}
	Since $A_{1,a}-B_{1,a} K$ is Hurwitz with eigenvalues of real part less than $-\delta$, we define $P \succ 0$ as the solution of the Lyapunov equation $(A_{1,a}-B_{1,a} K)^\top P + P (A_{1,a}-B_{1,a} K) + 2 \delta P + I = 0$. In the case of the $\mathcal{H}^0$-norm, we define
	\begin{equation*}
		V(\tilde{\mathcal{X}},u) = X_{1,a}^\top P X_{1,a} + c \sum_{n \geq N_{0}+1} \tilde{x}_{1,n}^2 + c \sum_{m \geq M_{0}+1} \tilde{x}_{2,m}^2 
	\end{equation*}
	for some constant $c > 0$ to be chosen later. We are going to show that $V$ is a Lyapunov function, thus yielding \eqref{eq: state-feedback - exp estimate}.
	The connection with the $\mathcal{H}^0$-norm lies in the fact that $\Psi$ is a Riesz basis of $\mathcal{H}^0$ (see Lemma~\ref{lem: Riesz basis}). For the study in $\mathcal{H}^1$-norm we define
$$
		V(\tilde{\mathcal{X}},u) =  X_{1,a}^\top P X_{1,a} + c \sum_{n \geq N_{0}+1} (n+1/2)^2 \tilde{x}_{1,n}^2  + c \sum_{m \geq M_{0}+1} (1+m^2) \tilde{x}_{2,m}^2 .
$$
	for some $c > 0$ to be chosen later. The connection with the $\mathcal{H}^1$-norm is provided by Lemma~\ref{lem: Riesz basis H1}. Indeed, from \eqref{eq: series expansion of the state}, one has
	\begin{align*}
		\tilde{\mathcal{X}}(t,\cdot) 
		& = (y(t,\cdot),\tilde{z}(t,\cdot)) 
		 = \sum_{n \geq 0} \tilde{x}_{1,n}(t) \phi_{1,n} + \sum_{m \geq 0} \tilde{x}_{2,m}(t) \phi_{2,m} \\
		& = \sum_{n \geq 0} ( n + 1/2 ) \pi \tilde{x}_{1,n}(t) \frac{1}{( n + 1/2 )\pi} \phi_{1,n}  + \sum_{m \geq 0} \sqrt{1+m^2\pi^2} \tilde{x}_{2,m}(t) \frac{1}{\sqrt{1+m^2\pi^2}} \phi_{2,m} .
	\end{align*}
This series converges in $\mathcal{H}^0$-norm. However, in view of Lemma~\ref{lem: Riesz basis H1}, the convergence also holds in $\mathcal{H}^1$-norm for classical solutions since $\tilde{\mathcal{X}}(t,\cdot) \in D(\mathcal{A})$. Due again to Lemma~\ref{lem: Riesz basis H1}, we deduce that $\Vert \tilde{\mathcal{X}}(t,\cdot) \Vert_{\mathcal{H}^1}^2$ is equivalent to $\sum_{n \geq 0} (n+1/2)^2\tilde{x}_{1,n}(t)^2 + \sum_{m \geq 0} (1+m^2)\tilde{x}_{2,m}(t)^2$. This motivates the above definition for $V$.
\\[2mm]
	The stability assessments for the $\mathcal{H}^0$-norm and the $\mathcal{H}^1$-norm are similar. We thus focus on the second one. Computing the time derivative along the closed-loop system trajectories consisting of \eqref{eq: augmented truncated model} and \eqref{eq: state-feedback}, \eqref{eq: dynamics tilde_x1n} and \eqref{eq: dynamics tilde_x2m}, using Young's inequality we infer that
	\begin{align*}
	\dot{V}+2\delta V 
	& = - \Vert X_{1,a} \Vert^2 + 2 c \sum_{n \geq N_{0}+1} (n+1/2)^2 \tilde{x}_{1,n} ( \dot{\tilde{x}}_{1,n} + \delta \tilde{x}_{1,n} )  + 2 c \sum_{m \geq M_{0}+1} (1+m^2) \tilde{x}_{2,m} ( \dot{\tilde{x}}_{2,m} + \delta \tilde{x}_{2,m} ) \\
	& \leq - \left( 1 - \epsilon c ( S_{\alpha} + \Vert K \Vert^2 S_{\beta} )  \right) \Vert X_{1,a} \Vert^2  + 2 c \sum_{n \geq N_{0}+1} (n+1/2)^2 \Gamma_{1,n} \tilde{x}_{1,n}^2  + 2 c \sum_{m \geq M_{0}+1} (1+m^2) \Gamma_{2,m} \tilde{x}_{2,m}^2
	\end{align*}
	for all $\epsilon > 0$, with $S_{\alpha} = \sum_{n \geq N_{0}+1} \alpha_{1,n}^2 + \sum_{m \geq M_{0}+1} \alpha_{2,m}^2$, $S_{\beta} = \sum_{n \geq N_{0}+1} \beta_{1,n}^2 + \sum_{m \geq M_{0}+1} \beta_{2,m}^2$, 
	\begin{align*}
	\Gamma_{1,n} & = \lambda_{1,n} + \delta + \frac{(n+1/2)^2}{\epsilon} = - \left( \pi^2 - \frac{1}{\epsilon} \right) (n+1/2)^2 + a + \delta , \\
	\Gamma_{2,m} & = \lambda_{2,m} + \delta + \frac{1+m^2}{\epsilon} = - \left( \pi^2 - \frac{1}{\epsilon} \right) m^2 + b + \delta + \frac{1}{\epsilon} .
	\end{align*}		
	For any $\epsilon > 1/\pi^2$, we note that $\Gamma_{1,n} \leq \Gamma_{1,N_{0}+1}$ for all $n \geq N_{0} + 1$ and $\Gamma_{2,m} \leq \Gamma_{2,M_{0}+1}$ for all $m \geq M_{0} + 1$. Hence, 
	\begin{multline*}
	\dot{V}+2\delta V \leq  - \left( 1 - \epsilon c ( S_{\alpha} + \Vert K \Vert^2 S_{\beta} )  \right) \Vert X_{1,a} \Vert^2  + 2 c \Gamma_{1,N_{0}+1} \sum_{n \geq N_{0}+1} (n+1/2)^2 \tilde{x}_{1,n}^2  \\
	 + 2 c \Gamma_{2,M_{0}+1} \sum_{m \geq M_{0}+1} (1+m^2) \tilde{x}_{2,m}^2
	\end{multline*}
	for any $\epsilon > 1/\pi^2$. 	Since $\lambda_{1,N_{0}+1}<-\delta$ and $\lambda_{2,M_{0}+1}<-\delta$, we fix $\epsilon > 1/\pi^2$ sufficiently large so that $\Gamma_{1,N_{0}+1} \leq 0$ and $\Gamma_{2,M_{0}+1} \leq 0$. Then we fix $c > 0$ small enough to ensure that $1 - \epsilon c ( S_{\alpha} + \Vert K \Vert^2 S_{\beta} ) \geq 0$. This implies that $\dot{V}+2\delta V \leq 0$, yielding \eqref{eq: state-feedback - exp estimate}. \qed
\end{proof}

\begin{remark}
Based on Remark~\ref{rem: system 2}, the same result as stated by Theorem~\ref{thm: state-feedback} holds true for \eqref{eq: system 2} when removing the controllability assumption \eqref{eq: controllability assumption}.
\end{remark}

\section{Output-feedback}\label{sec: output-feedback}
In this section, we extend the result of the previous section to the case of output-feedback control with measurement done on the second component of the PDE cascade \eqref{eq: system}.

\subsection{Distributed output operator}
We consider the system output 
\begin{equation}\label{eq: output}
	y_m(t) = \int_0^1 c(x) y(t,x) \,\mathrm{d}x
\end{equation}
for some $c \in L^2(0,1)$. Defining $c_{1,n} = \int_0^1 c(x) \phi_{1,n}^1(x) \,\mathrm{d}x$ and $c_{2,m} = \int_0^1 c(x) \phi_{2,m}^1(x) \,\mathrm{d}x$, we infer from \eqref{eq: series expansion of the state} that
\begin{equation}
	y_m(t)  = \sum_{n \geq 0} c_{1,n} \tilde{x}_{1,n}(t) + \sum_{m \geq 0} c_{2,m} \tilde{x}_{2,m}(t) 
	 = C_1 X_1(t) + \sum_{n \geq N_{0} + 1} c_{1,n} \tilde{x}_{1,n}(t)  + \sum_{m \geq M_{0} + 1} c_{2,m} \tilde{x}_{2,m}(t) \label{eq: output series expansion}
\end{equation}
where
$
	C_1 = \begin{bmatrix} C_{1,1} & C_{1,2} & C_{1,3} \end{bmatrix}
$
with
\begin{align*}
C_{1,1} & = \begin{bmatrix} c_{1,\varphi_1(0)} & c_{2,\varphi_2(0)} & \ldots & c_{1,\varphi_1(N_\Delta-1)} & c_{2,\varphi_2(N_\Delta-1)} \end{bmatrix} , \\
C_{1,2} & = \begin{bmatrix} c_{1,\varphi_1(N_\Delta)} & \ldots & c_{1,\varphi_1(N_{0})} \end{bmatrix} , \\
C_{1,3} & = \begin{bmatrix} c_{2,\varphi_2(N_\Delta)} & \ldots & c_{2,\varphi_2(M_{0})} \end{bmatrix} .
\end{align*}
The Hautus test gives the following result. 
\begin{lemma}\label{lem: obsv}
$(A_1,C_1)$ is observable if and only if 
\begin{subequations}\label{eq: obsv condition}
\begin{align}
& c_{1,n} \neq 0 , \quad  \forall 0 \leq n \leq N_{0} , \\ 
& c_{2,m} \neq 0 , \quad \forall m\in\{0,1,\ldots,M_0\}\backslash\Delta_2(a,b) .
\end{align}
\end{subequations}
\end{lemma}

We now define our output-feedback control strategy. Let $N \geq N_{0}$ and $M \geq M_{0}$ be integers to be chosen later. We introduce the following controller dynamics:
\begin{equation}\label{eq: output-feedback controller}
\begin{split}
	\dot{u} & = v \\
	\dot{\hat{X}}_1 & = A_1 \hat{X}_1 + B_{1,u} u + B_{1,v} v - L \left\{ C_1 \hat{X}_1 + \sum_{n=N_{0}+1}^{N} c_{1,n} \hat{x}_{1,n}  + \sum_{m=M_{0}+1}^{M} c_{2,m} \hat{x}_{2,m} - y_m(t) \right\} \\
	\dot{\hat{x}}_{1,n} & = \lambda_{1,n} \hat{x}_{1,n} + \alpha_{1,n} u + \beta_{1,n} v , \; N_{0} +1 \leq n \leq N \\
	\dot{\hat{x}}_{2,m} & = \lambda_{2,m} \hat{x}_{2,m} + \alpha_{2,m} u + \beta_{2,m} v , \; M_{0} +1 \leq m \leq M \\
	v & = - K_x \hat{X}_1 -k_u u
\end{split}
\end{equation}
where $k_u \in \mathbb{R}$, $K_x \in \mathbb{R}^{1 \times (N_{0}+M_{0}+2)}$ are feedback gains while $L \in \mathbb{R}^{N_{0}+M_{0}+2}$ is the observation matrix. This controller architecture is inspired by the seminal work \cite{sakawa1983feedback}, which has been extended in a number of ways (see \cite{grune2021finite,katz2020constructive,lhachemi2020finite,lhachemi2021nonlinear}).

\begin{theorem}\label{thm2}
	Let $a,b,s \in\mathbb{R}$ with $s \neq 0$ and $c \in L^2(0,1)$. Let $\delta > 0$ be arbitrary. Let $N_{0},M_{0} \geq 0$ be integers such that $\lambda_{1,N_{0}+1} < -\delta$, $\lambda_{2,M_{0}+1} < -\delta$, $N_0 \geq \max\Delta_1(a,b)$ and $M_0 \geq \max\Delta_2(a,b)$. Assuming that both controllability and observability assumptions \eqref{eq: controllability assumption} and \eqref{eq: obsv condition} hold, let $K = \begin{bmatrix} K_x & k_u \end{bmatrix}\mathbb{R}^{1 \times (N_{0}+M_{0}+3)}$ and $L\in\mathbb{R}^{N_{0}+M_{0}+2}$ be matrices such that $A_{1,a} - B_{1,a} K$ and $A_1 - L C_1$ are Hurwitz with eigenvalues of real part less than $-\delta$. Then, for any integers $N \geq N_{0}$ and $M \geq M_{0}$ chosen sufficiently large, there exists $C > 0$ such that, for all initial conditions $y_0,z_0 \in L^2(0,1)$, $u(0)=u_0\in\mathbb{R}$, $\hat{X}_1(0)\in\mathbb{R}^{N_{0}+M_{0}+2}$, and $\hat{x}_{1,n}(0),\hat{x}_{2,m}(0)\in\mathbb{R}$, the solutions of the system \eqref{eq: system} in closed-loop with the output-feedback \eqref{eq: output-feedback controller} 
satisfy
	\begin{align}
		& \Vert ( y(t,\cdot) , z(t,\cdot) ) \Vert_{Y}^2 + u(t)^2 + \Vert \hat{X}_1(t) \Vert^2  + \sum_{n=N_{0}+1}^{N} \hat{x}_{1,n}(t)^2 + \sum_{m=M_{0}+1}^{M} \hat{x}_{2,m}(t)^2 \nonumber \\
		& \qquad \leq C e^{-2\delta t} \Bigg( \Vert ( y_0 , z_0 ) \Vert_{Y}^2 + u_0^2 + \Vert \hat{X}_1(0) \Vert^2 + \sum_{n=N_{0}+1}^{N} \hat{x}_{1,n}(0)^2 + \sum_{m=M_{0}+1}^{M} \hat{x}_{2,m}(0)^2 \Bigg) \label{eq: exp stab estimate}
	\end{align}
for every $t \geq 0$, with $Y = \mathcal{H}^0$ defined by \eqref{eq: space H0}. If the initial conditions are such that $y_0,z_0\in H^2(0,1)$ with $y_0(1)=z_0'(1)=0$, $y_0'(0)=sz_0(0)$, and $z_0'(0)=u(0)$, then \eqref{eq: exp stab estimate} holds with $Y = \mathcal{H}^1$ defined by \eqref{eq: space H1}.
\end{theorem}

\begin{proof}
Define the observation errors $e_{1,n} = \tilde{x}_{1,n} - \hat{x}_{1,n}$ and $e_{2,m} = \tilde{x}_{2,m} - \hat{x}_{2,m}$, and the vectors
\begin{subequations}
\begin{align}
	\hat{X}_{1,a} & = \mathrm{col}( \hat{X}_1 , u ) , \nonumber \\ 
	E_{1,1} & = \begin{bmatrix} e_{1,\varphi_1(0)} & e_{2,\varphi_2(0)} &  \ldots & e_{1,\varphi_1(N_\Delta-1)} & e_{2,\varphi_2(N_\Delta-1)} \end{bmatrix}^\top , \nonumber \\
	E_{1,2} & = \begin{bmatrix} e_{1,\varphi_1(N_\Delta)} & \ldots & e_{1,\varphi_1(N_{0})} \end{bmatrix}^\top , \quad
	E_{1,3}  = \begin{bmatrix} e_{2,\varphi_2(N_\Delta)} & \ldots & e_{2,\varphi_2(M_{0})} \end{bmatrix}^\top , \nonumber \\
	E_1 & = \mathrm{col}( E_{1,1} , E_{1,2} , E_{1,3} ) ,\nonumber  \\
	\hat{X}_{2,1} & = \begin{bmatrix} \hat{x}_{1,N_{0}+1} & \ldots & \hat{x}_{1,N} \end{bmatrix}^\top , \quad
	\hat{X}_{2,2}  = \begin{bmatrix} \hat{x}_{2,M_{0}+1} & \ldots & \hat{x}_{2,M} \end{bmatrix}^\top , \quad
	\hat{X}_{2}  = \mathrm{col}( \hat{X}_{2,1} , \hat{X}_{2,2} ) ,	 \nonumber \\
	E_{2,1} & = \begin{bmatrix} (N_0+2)^{\kappa_1} e_{1,N_{0}+1} & \ldots & (N+1)^{\kappa_1} e_{1,N} \end{bmatrix}^\top , \label{eq: def E21}\\
	E_{2,2} & = \begin{bmatrix} (M_{0}+2)^{\kappa_2} e_{2,M_{0}+1} & \ldots & (M+1)^{\kappa_2} e_{2,M} \end{bmatrix}^\top , \label{eq: def E22}\\
	E_{2} & = \mathrm{col}( E_{2,1} , E_{2,2} ) , \quad
	X  = \mathrm{col}(\hat{X}_{1,a},E_1,\hat{X}_2,E_2) . \nonumber 
\end{align}
\end{subequations}
In this proof we set $\kappa_1 = \kappa_2 = 0$. Different values of $\kappa_1,\kappa_2 \geq 0$ will be chosen in the proof of Theorem~\ref{thm3} given later. Using \eqref{eq: dynamics tilde_x1n}, \eqref{eq: dynamics tilde_x2m}, \eqref{eq: dynamics X1}, \eqref{eq: augmented truncated model}, \eqref{eq: output series expansion} and \eqref{eq: output-feedback controller}, we infer that 
\begin{equation}\label{eq: output-feedback - finite dim model}
	\dot{X} = F X + \mathcal{L} \zeta_1 + \mathcal{L} \zeta_2
\end{equation}
with the measurement residues
\begin{align}\label{eq: residues of measurement}
\zeta_1 = \sum_{n \geq N + 1} c_{1,n} \tilde{x}_{1,n} , \quad
\zeta_2 = \sum_{m \geq M + 1} c_{2,m} \tilde{x}_{2,m},
\end{align}
where
\begin{equation*}
F = \begin{bmatrix}
A_{1,a} - B_{1,a} K & L_a C_1 & 0 & L_a C_2 \\
0 & A_1-LC_1 & 0 & -L C_2 \\
\left[ 0 \; B_{2,u} \right] - B_{2,v} K & 0 & A_2 & 0 \\
0 & 0 & 0 & A_2
\end{bmatrix} , \;
\mathcal{L} = \begin{bmatrix}
L_a \\ - L \\ 0 \\ 0
\end{bmatrix}
\end{equation*}
and 
\begin{align}
K & = \begin{bmatrix} K_x & k_u \end{bmatrix} , \quad
L_a = \mathrm{col}(L,0) , \nonumber \\
A_2 & = \mathrm{diag}(\lambda_{1,N_{0}+1},\ldots,\lambda_{1,N},\lambda_{2,M_{0}+1},\ldots,\lambda_{2,M}) , \nonumber \\
B_{2,u} & = \begin{bmatrix} \alpha_{1,N_{0}+1} & \ldots & \alpha_{1,N} & \alpha_{2,M_{0}+1} & \ldots & \alpha_{2,M} \end{bmatrix}^\top , \nonumber \\
B_{2,v} & = \begin{bmatrix} \beta_{1,N_{0}+1} & \ldots & \beta_{1,N} & \beta_{2,M_{0}+1} & \ldots & \beta_{2,M} \end{bmatrix}^\top , \nonumber \\
C_{2} & = \begin{bmatrix} \frac{c_{1,N_{0}+1}}{(N_0+2)^{\kappa_1}} & \ldots & \frac{c_{1,N}}{(N+1)^{\kappa_1}} & \frac{c_{2,M_{0}+1}}{(M_{0}+2)^{\kappa_2}} & \ldots & \frac{c_{2,M}}{(M+1)^{\kappa_2}} \end{bmatrix} . \label{eq: def C2}
\end{align}
We finally note that
\begin{equation*}
u = E \hat{X}_{1,a} = \tilde{E} X , \quad
v = - K \hat{X}_{1,a} = - \tilde{K} X
\end{equation*}
where $E = \begin{bmatrix} 0 & \ldots & 0 & 1 \end{bmatrix}$, $\tilde{E} = \begin{bmatrix} E & 0 & 0 & 0 \end{bmatrix}$ and $\tilde{K} = \begin{bmatrix} K & 0 & 0 & 0 \end{bmatrix}$.

In the case of the $\mathcal{H}^0$-norm, we define (see Lemma~\ref{lem: Riesz basis})
	\begin{equation*}
		V = X^\top P X + \sum_{n \geq N_{0}+1} \tilde{x}_{1,n}^2 + \sum_{m \geq M_{0}+1} \tilde{x}_{2,m}^2 
	\end{equation*}
	for some $P \succ 0$. For the study in $\mathcal{H}^1$-norm we define (see Lemma~\ref{lem: Riesz basis H1})
	\begin{equation*}
		V  = X^\top P X + \sum_{n \geq N+1} (n+1/2)^2 \tilde{x}_{1,n}^2 + \sum_{m \geq M+1} (1+m^2) \tilde{x}_{2,m}^2 
	\end{equation*}
	for some $P \succ 0$. We focus on the  second case. Computing the time derivative along the solutions of \eqref{eq: dynamics tilde_x1n}, \eqref{eq: dynamics tilde_x2m} and \eqref{eq: output-feedback - finite dim model}, using Young's inequality, we have
	\begin{align*}
	\dot{V}+2\delta V 
	&= \tilde{X}^\top \begin{bmatrix} F^\top P + P F + 2 \delta P & P\mathcal{L} & P\mathcal{L} \\  \mathcal{L}^\top P & 0 & 0 \\  \mathcal{L}^\top P & 0 & 0 \end{bmatrix}	 \tilde{X}  + 2 \sum_{n \geq N+1} (n+1/2)^2 ( \lambda_{1,n} + \delta ) \tilde{x}_{1,n}^2 \\
	& \phantom{=}\; + 2 \sum_{n \geq N+1} (n+1/2)^2 \tilde{x}_{1,n} ( \alpha_{1,n} u + \beta_{1,n} v )  + 2 \sum_{m \geq M+1} (1+m^2) (\lambda_{2,m}+\delta) \tilde{x}_{2,m}^2 \\
	& \phantom{=}\; + 2 \sum_{m \geq M+1} (1+m^2) \tilde{x}_{2,m} ( \alpha_{2,m} u + \beta_{2,m} v ) \\
	& \leq \tilde{X}^\top \begin{bmatrix} \Theta_{1,1} & P\mathcal{L} & P\mathcal{L} \\  \mathcal{L}^\top P & 0 & 0 \\  \mathcal{L}^\top P & 0 & 0 \end{bmatrix}	 \tilde{X} + 2 \sum_{n \geq N+1} (n+1/2)^2 \left( \lambda_{1,n} + \frac{(n+1/2)^2}{\epsilon_1} + \delta \right) \tilde{x}_{1,n}^2 \\
	& \phantom{\leq}\; + 2 \sum_{m \geq M+1} (1+m^2) \left( \lambda_{2,m} + \frac{1+m^2}{\epsilon_2} + \delta \right) \tilde{x}_{2,m}^2 
	\end{align*}
	for all $\epsilon_1,\epsilon_2 > 0$ with $\tilde{X} = \mathrm{col}(X,\zeta_1,\zeta_2)$, 
	$$
	\Theta_{1,1}  = F^\top P + P F + 2 \delta P + \epsilon_1 \left( S_{\alpha_1,N} \tilde{E}^\top \tilde{E} + S_{\beta_1,N} \tilde{K}^\top \tilde{K} \right)  + \epsilon_2 \left( S_{\alpha_2,M} \tilde{E}^\top \tilde{E} + S_{\beta_2,M} \tilde{K}^\top \tilde{K} \right)
	$$
	where $S_{\alpha_1,N} = \sum_{n \geq N +1} \alpha_{1,n}^2$, $S_{\alpha_2,M} = \sum_{m \geq M +1} \alpha_{2,m}^2$, $S_{\beta_1,N} = \sum_{n \geq N +1} \beta_{1,n}^2$, and $S_{\beta_2,M} = \sum_{m \geq M +1} \beta_{2,m}^2$. We infer from \eqref{eq: residues of measurement} and from the Cauchy-Schwarz inequality that 
	$$
		\zeta_1^2  \leq \underbrace{\sum_{n \geq N+1} c_{1,n}^2}_{= S_{\zeta_1,N} < \infty} \sum_{n \geq N+1} \tilde{x}_{1,n}^2 , \qquad
		\zeta_2^2  \leq \underbrace{\sum_{m \geq M+1} c_{2,m}^2}_{= S_{\zeta_2,M} < \infty} \sum_{m \geq M+1} \tilde{x}_{2,m}^2 .
	$$	
Hence
	\begin{equation}
	\dot{V}+2\delta V 
	 \leq \tilde{X}^\top \Theta_1 \tilde{X} + \sum_{n \geq N+1} (n+1/2)^2 \Gamma_{1,n}\tilde{x}_{1,n}^2 + \sum_{m \geq M+1} (1+m^2) \Gamma_{2,m} \tilde{x}_{2,m}^2 \label{eq: dotV}
	\end{equation}
	where
	\begin{subequations}
	\begin{align}
	\Theta_1 & = \begin{bmatrix} \Theta_{1,1} & P\mathcal{L} & P\mathcal{L} \\  \mathcal{L}^\top P & -\eta_1 & 0 \\  \mathcal{L}^\top P & 0 & -\eta_2 \end{bmatrix} , \label{eq: Theta_1} \\
	\Gamma_{1,n} & = 2 \left( \lambda_{1,n} + \frac{(n+1/2)^2}{\epsilon_1} + \delta \right) + \frac{\eta_1 S_{\zeta_1,N}}{(n+1/2)^2} \nonumber \\
	& = 2 \left( - (n+1/2)^2 \left( \pi^2 - \frac{1}{\epsilon_1} \right) + a + \delta \right) + \frac{\eta_1 S_{\zeta_1,N}}{(n+1/2)^2} \label{eq: Gamma_1_n} , \\
	\Gamma_{2,m} & = 2 \left( \lambda_{2,m} + \frac{1+m^2}{\epsilon_2} + \delta \right) + \frac{\eta_2 S_{\zeta_2,M}}{1+m^2} \nonumber \\
	& = 2 \left( - m^2 \left( \pi^2 - \frac{1}{\epsilon_2} \right) + b + \delta + \frac{1}{\epsilon_2} \right) + \frac{\eta_2 S_{\zeta_2,M}}{1+m^2} \label{eq: Gamma_2_m}
	\end{align}
	\end{subequations}
	for all $\epsilon_1,\epsilon_2,\eta_1,\eta_2 > 0$. Hence, for all $\epsilon_1,\epsilon_2 > 1/\pi^2$, we have $\Gamma_{1,n} \leq \Gamma_{1,N+1}$ for all $n \geq N+1$ and $\Gamma_{2,m} \leq \Gamma_{2,M+1}$ for all $m \geq M+1$. This gives
	\begin{equation*}
	\dot{V}+2\delta V 
	 \leq \tilde{X}^\top \Theta_1 \tilde{X} + \Gamma_{1,N+1} \sum_{n \geq N+1} (n+1/2)^2 \tilde{x}_{1,n}^2  + \Gamma_{2,M+1} \sum_{m \geq M+1} (1+m^2) \tilde{x}_{2,m}^2
	\end{equation*}
	for all $\epsilon_1,\epsilon_2 > 1/\pi^2$ and $\eta_1,\eta_2 > 0$. Therefore, $\dot{V}+2\delta V \leq 0$, yielding \eqref{eq: exp stab estimate}, provided that we can find $N \geq N_0 + 1$, $M \geq M_0 + 1$, $\epsilon_1,\epsilon_2 > 1/\pi^2$, $\eta_1,\eta_2 > 0$ and $P \succ 0$ such that
	\begin{equation*}
	\Theta_1 \preceq 0 , \quad
	 \Gamma_{1,N+1} \leq 0 , \quad
	 \Gamma_{2,M+1} \leq 0 .
	\end{equation*}
	Let us prove the feasibility of these constraints. Since $F$ is Hurwitz with eigenvalues of real part less than $-\delta$, and since $\Vert C_2 \Vert = \mathrm{O}(1)$, $\Vert B_{2,u} \Vert = \mathrm{O}(1)$ and $\Vert B_{2,v} \Vert = \mathrm{O}(1)$ as $N,M \rightarrow + \infty$, the application of \cite[Lemma in Appendix]{lhachemi2020finite} shows that the solution $P \succ 0$ to $F^\top P + P F + 2 \delta P = -I$ is such that $\Vert P \Vert = \mathrm{O}(1)$ as $(N,M) \rightarrow + \infty$. Then, fixing $\epsilon_1,\epsilon_2 > 1/\pi^2$ and 
	\begin{equation*}
\eta_1 = \left\{\begin{array}{cl}
\frac{1}{\sqrt{S_{\zeta_1,N}}} & \mathrm{if}\; S_{\zeta_1,N} \neq 0 \\
N & \mathrm{otherwise}
\end{array}\right.
\qquad
\eta_2 = \left\{\begin{array}{cl}
\frac{1}{\sqrt{S_{\zeta_2,M}}} & \mathrm{if}\; S_{\zeta_2,M} \neq 0 \\
M & \mathrm{otherwise}
\end{array}\right.
\end{equation*}
	we deduce that $\Gamma_{1,N+1} \rightarrow - \infty$ as $N \rightarrow +\infty$, $\Gamma_{2,M+1} \rightarrow - \infty$ as $M \rightarrow +\infty$, and $\Theta_1 \preceq 0$ for $N,M$ large enough. 
	The last claim follows from the Schur complement based on the facts that $\Vert \tilde{E} \Vert = 1$, $\Vert \tilde{K} \Vert = \Vert K \Vert$, and $\Vert \mathcal{L} \Vert = \sqrt{2} \Vert L \Vert$ are constants independent of $N,M$, $\Vert P \Vert = \mathrm{O}(1)$ as $(N,M) \rightarrow + \infty$, and $S_{\alpha_1,N},S_{\beta_1,N} \rightarrow 0$ as $N \rightarrow +\infty$ and $S_{\alpha_2,M},S_{\beta_2,M} \rightarrow 0$ as $M \rightarrow +\infty$. This completes the proof. \qed
\end{proof}

\subsection{Extension to pointwise measurements}
In this section, we extend the result of the previous section to the case of pointwise (possibly, Dirichlet or Neumann) measurement operators, namely, either
\begin{equation}\label{eq: Dirichlet measurement}
	y_m(t) = y(t,\xi_p) 
\end{equation}
or
\begin{equation}\label{eq: Neumann measurement}
	y_m(t) = \partial_x y(t,\xi_p) 
\end{equation}
for some $\xi_p\in[0,1]$. Hence, defining $c_{1,n} = \phi_{1,n}^1(\xi_p)$ and $c_{2,m} = \phi_{2,m}^1(\xi_p)$ in the case of the measurement \eqref{eq: Dirichlet measurement}, and $c_{1,n} = (\phi_{1,n}^1)'(\xi_p)$ and $c_{2,m} = (\phi_{2,m}^1)'(\xi_p)$ in the case of the measurement \eqref{eq: Neumann measurement}, we infer from \eqref{eq: series expansion of the state} that the series expansion \eqref{eq: output series expansion} holds.
\\[2mm]
In view of Lemma~\ref{lem: obsv}, the pair $(A_1,C_1)$ is observable if and only if:
\begin{itemize}
\item For the measurement \eqref{eq: Dirichlet measurement}: 
\begin{subequations}\label{eq: obsv cond for Dirichlet}
\begin{align}
& \cos((n+1/2)\pi\xi_p) \neq 0 , \quad \forall 0 \leq n \leq N_0 \label{eq: obsv cond for Dirichlet-1} \\
& \left\{\begin{array}{l}
\xi_p -1 \neq 0 ,\,\mathrm{if}\, b-a = m^2 \pi^2 \\
\sinh(r_m(1-\xi_p)) \neq 0 ,\, \mathrm{if}\, b-a \neq m^2 \pi^2
\end{array}\right. \nonumber , \quad\forall m\in\{0,\ldots,M_0\}\backslash\Delta_2(a,b) \label{eq: obsv cond for Dirichlet-2}
\end{align}
\end{subequations}
\item For the measurement \eqref{eq: Neumann measurement}:
\begin{subequations}\label{eq: obsv cond for Neumann}
\begin{align}
\sin((n+1/2)\pi\xi_p) \neq 0 , \quad & \forall 0 \leq n \leq N_0 \\
\cosh(r_m(1-\xi_p)) \neq 0 , \quad & \forall m\in\{0,\ldots,M_0\}\backslash\Delta_2(a,b) \quad \mathrm{s.t.}\; b-a \neq m^2 \pi^2
\end{align}
\end{subequations}
\end{itemize}
where in both cases $r_m$ is one of the two square roots of $\lambda_{2,m}-a=b-a-m^2\pi^2$.

\begin{theorem}\label{thm3}
	Considering either the measurement \eqref{eq: Dirichlet measurement} or the measurement \eqref{eq: Neumann measurement} for some $\xi_p \in [0,1]$, the same statement as Theorem~\ref{thm2} holds true for the Hilbert space $Y = \mathcal{H}^1$ defined by \eqref{eq: space H1}.
\end{theorem} 

\begin{proof}
The proof follows the one of Theorem~\ref{thm2}, with two key differences. The first concerns the selection of the constants $\kappa_1,\kappa_2$ involved in the definitions of the observation error vectors $E_{2,1},E_{2,2}$ given by \eqref{eq: def E21}-\eqref{eq: def E22} and of the matrix $C_2$ defined by \eqref{eq: def C2}. The objective is to ensure that $\Vert C_2 \Vert = \mathrm{O}(1)$ as $N,M \rightarrow + \infty$. This is crucial in the proof to ensure that the solution $P \succ 0$ to the Lyapunov equation $F^\top P + P F + 2 \delta P = -I$ satisfies $\Vert P \Vert = \mathrm{O}(1)$ as $N,M \rightarrow +\infty$ (see \cite[Lemma in Appendix]{lhachemi2020finite}). The second difference is about the estimation of the measurement residues $\zeta_1,\zeta_2$ defined by \eqref{eq: residues of measurement}. 
\\[2mm]
In the case of the measurement \eqref{eq: Dirichlet measurement}, since $c_{1,n} = \phi_{1,n}^1(\xi_p) = \mathrm{O}(1)$ as $n \rightarrow +\infty$ and $c_{2,m} = \phi_{2,m}^1(\xi_p) = \mathrm{O}(1/m)$ as $m \rightarrow +\infty$, we set $\kappa_1 = 1$ and $\kappa_2 =0$. This ensures in view of \eqref{eq: def C2} that $\Vert C_2 \Vert = \mathrm{O}(1)$ as $N,M \rightarrow + \infty$. Moreover, we have
\begin{equation*}
	\zeta_1^2  \leq \underbrace{\sum_{n \geq N+1} \frac{c_{1,n}^2}{(n+1/2)^2}}_{= S_{\zeta_1,N} < \infty} \sum_{n \geq N+1} (n+1/2)^2 \tilde{x}_{1,n}^2 , \qquad
	\zeta_2^2  \leq \underbrace{\sum_{m \geq M+1} c_{2,m}^2}_{= S_{\zeta_2,M} < \infty} \sum_{m \geq M+1} \tilde{x}_{2,m}^2 .
\end{equation*}	
Hence \eqref{eq: dotV} holds with
\begin{equation*}
	\Gamma_{1,n}  = 2 \left( \lambda_{1,n} + \frac{(n+1/2)^2}{\epsilon_1} + \delta \right) + \eta_1 S_{\zeta_1,N} 
	 = 2 \left( - (n+1/2)^2 \left( \pi^2 - \frac{1}{\epsilon_1} \right) + a + \delta \right) + \eta_1 S_{\zeta_1,N} 
\end{equation*}
and $\Theta_1$ and $\Gamma_{2,m}$ defined by \eqref{eq: Theta_1} and \eqref{eq: Gamma_2_m} respectively.
\\[2mm]
In the case of the measurement \eqref{eq: Neumann measurement}, we have $c_{1,n} = (\phi_{1,n}^1)'(\xi_p) = \mathrm{O}(n)$ as $n \rightarrow +\infty$ and $c_{2,m} = (\phi_{2,m}^1)'(\xi_p) = \mathrm{O}(1)$ as $m \rightarrow +\infty$. Hence, we set $\kappa_1 = 7/4$ and $\kappa_2 = 1$, which ensures, using \eqref{eq: def C2}, that $\Vert C_2 \Vert = \mathrm{O}(1)$ as $N,M \rightarrow + \infty$. We also have 
\begin{equation*}
	\zeta_1^2  \leq \underbrace{\sum_{n \geq N+1} \frac{c_{1,n}^2}{(n+1/2)^{7/2}}}_{= S_{\zeta_1,N} < \infty} \sum_{n \geq N+1} (n+1/2)^{7/2} \tilde{x}_{1,n}^2 , \qquad
	\zeta_2^2  \leq \underbrace{\sum_{m \geq M+1} \frac{c_{2,m}^2}{1+m^2}}_{= S_{\zeta_2,M} < \infty} \sum_{m \geq M+1} (1+m^2) \tilde{x}_{2,m}^2 .
\end{equation*}	
Hence \eqref{eq: dotV} holds with
\begin{align*}
	\Gamma_{1,n} & = 2 \left( \lambda_{1,n} + \frac{(n+1/2)^2}{\epsilon_1} + \delta \right) + (n+1/2)^{3/2} \eta_1 S_{\zeta_1,N} \\
	& = 2 \left( - (n+1/2)^2 \left( \pi^2 - \frac{1}{\epsilon_1} \right) + a + \delta \right) + (n+1/2)^{3/2} \eta_1 S_{\zeta_1,N}  \\
	\Gamma_{2,m} & = 2 \left( \lambda_{2,m} + \frac{1+m^2}{\epsilon_2} + \delta \right) + \eta_2 S_{\zeta_2,M} 
	 = 2 \left( - m^2 \left( \pi^2 - \frac{1}{\epsilon_2} \right) + b + \delta + \frac{1}{\epsilon_2} \right) + \eta_2 S_{\zeta_2,M}
\end{align*}
with $\Theta_1$ defined by \eqref{eq: Theta_1}.	
\\[2mm]
In both cases, the remainder of the proof now follows the arguments of the proof of Theorem~\ref{thm2}. \qed
\end{proof}

\begin{remark}
In the case of the measurement \eqref{eq: Dirichlet measurement}, the condition \eqref{eq: obsv cond for Dirichlet-1} is the usual one characterizing the observability of the mode $\lambda_{1,n}$ of the heat equation \eqref{eq: system - y} with boundary conditions $\partial_x y(t,0) = y(t,1)$. The condition \eqref{eq: obsv cond for Dirichlet-2} captures the possibility to observe the mode $\lambda_{2,m}$, associated with the first PDE component \eqref{eq: system - z} of the cascade, through a measurement done on the second component \eqref{eq: system - y} of the cascade. Similar remarks can be done to the distributed measurement \eqref{eq: output} and to the pointwise measurement \eqref{eq: Neumann measurement}.
\end{remark}

\begin{remark}
The observability conditions \eqref{eq: obsv condition} for the distributed measurement \eqref{eq: output}, the conditions \eqref{eq: obsv cond for Dirichlet} for the pointwise measurement \eqref{eq: Dirichlet measurement}, and the conditions \eqref{eq: obsv cond for Neumann} for the pointwise measurement \eqref{eq: Neumann measurement}, do not depend on the location of the input. Hence, the results discussed in this section apply to \eqref{eq: system 2}. In particular, since its modes are always controllable (see Remark~\ref{rem: system 2}), both Theorems~\ref{thm2} and~\ref{thm3} apply to the PDE cascade \eqref{eq: system 2} when removing the controllability assumption \eqref{eq: controllability assumption}.
\end{remark}

\section{Dual problems}\label{sec: Dual problems}
Owing to the definition of the adjoint operator \eqref{eq: adjoint operator}, the approaches developed in this paper also apply to the two following heat-heat cascades: 
\begin{subequations}\label{eq: system 3}
	\begin{align}
		& \partial_t y = \partial_{xx} y + a y , &&\\
		& \partial_t z = \partial_{xx} z + b z , &&\\
		& \partial_x z(t,0) = s y(t,0) , && y(t,1) = \partial_x z(t,1) = 0 , \\
		& \partial_x y(t,0) = u(t) , \\
		& y(0,x) = y_0(x) , && z(0,x) = z_0(x),
	\end{align}
\end{subequations}
and
\begin{subequations}\label{eq: system 4}
	\begin{align}
		& \partial_t y = \partial_{xx} y + a y , && \\
		& \partial_t z = \partial_{xx} z + b z , && \\
		& \partial_x z(t,0) = s y(t,0) , && \partial_x y(t,0) = \partial_x z(t,1) = 0 , \\
		& y(t,1) = u(t) , \\
		& y(0,x) = y_0(x) , && z(0,x) = z_0(x) .
	\end{align}
\end{subequations}
We summarize in this section the key elements for the study of these two PDE cascades that differ from the previously studied ones. We focus our presentation on the system \eqref{eq: system 3} and discuss the differences for the system \eqref{eq: system 4} in some remarks.

\subsection{Spectral analysis}
In the two PDE cascades \eqref{eq: system 3} and \eqref{eq: system 4}, the underlying unbounded operator is $\mathcal{A}^*$ defined by \eqref{eq: adjoint operator}, whose spectral properties are described in Section~\ref{sec: spectral analysis}. In particular, its set of generalized eigenvectors $\Psi$, defined in Lemma~\ref{eq: dual Riesz basis}, is a Riesz basis of the Hilbert space $\mathcal{H}^0$ defined by \eqref{eq: space H0}, of dual Riesz basis $\Phi$ defined by Lemma~\ref{lem: eigenelements A}. In order to study the solutions also in $H^1$-norm, we give the following lemma, whose proof follows the same arguments as the one of Lemma~\ref{lem: Riesz basis H1}.

\begin{lemma}
	The set $\Psi^1 = \{ \frac{1}{(n+1/2)\pi}\psi_{1,n} \,\mid\, n\geq 0 \} \cup \{ \frac{1}{\sqrt{1+m^2\pi^2}}\psi_{2,m} \,\mid\, m \geq 0\}$ is a Riesz basis of the Hilbert space $\mathcal{H}^1$ defined by \eqref{eq: space H1}.
\end{lemma}

\subsection{Spectral reduction}

\subsubsection{Homogeneous representation}
The state is $\mathcal{X}(t) = (y(t,\cdot),z(t,\cdot))$.
We set $\varphi(x) = x-x^2$, so that $\varphi(0) = \varphi(1) = 0$ and $\varphi'(0) = 1$, and we make the change of variable 
\begin{equation}\label{eq: change of variable bis}
	\tilde{y}(t,x) = y(t,x) - \varphi(x) u(t) .
\end{equation}
In view of the dynamics \eqref{eq: system 3}, we have
	\begin{align*}
		& \partial_t \tilde{y} = \partial_{xx} \tilde{y} + a \tilde{y} + \alpha u + \beta \dot{u} ,\\
		& \partial_t z = \partial_{xx} z + b z, \\
		& \partial_x z(t,0) = s \tilde{y}(t,0) ,\ \ \tilde{y}(t,1) = \partial_x \tilde{y}(t,0) = \partial_x z(t,1) = 0 ,
	\end{align*}
with $\alpha(x) = \varphi''(x) + a \varphi(x)$ and $\beta(x) = - \varphi(x)$. Considering the state $\tilde{\mathcal{X}}(t) = (\tilde{y}(t,\cdot),z(t,\cdot))$ and $v = \dot{u}$, we have
\begin{equation}\label{eq: abstract system homogeneous bis}
	\frac{\mathrm{d}\tilde{\mathcal{X}}}{dt} = \mathcal{A}^* \tilde{\mathcal{X}} + (\alpha,0) u + (\beta,0) v .
\end{equation}
Note that $\tilde{\mathcal{X}} = \mathcal{X} + (\beta,0) u$.

\subsubsection{Dynamics of the modes}
While the study of the two PDE cascades \eqref{eq: system} and \eqref{eq: system 2} was based on the expansion of the solutions in the Riesz basis $\Phi$, the study of the two dual PDE cascades \eqref{eq: system 3} and \eqref{eq: system 4} requires the expansion of the solutions in the dual Riesz basis $\Psi$. We thus define the coefficients $\tilde{x}_{1,n} = \langle \tilde{\mathcal{X}} , \phi_{1,n} \rangle$, $\tilde{x}_{2,m} = \langle \tilde{\mathcal{X}} , \phi_{2,m} \rangle$, $x_{1,n} = \langle \mathcal{X} , \phi_{1,n} \rangle$, $\quad x_{2,m} = \langle \mathcal{X} , \phi_{2,m} \rangle$, $\alpha_{1,n} = \langle (\alpha,0) , \phi_{1,n} \rangle$, $\alpha_{2,m} = \langle (\alpha,0) , \phi_{2,m} \rangle$, $\beta_{1,n} = \langle (\beta,0) , \phi_{1,n} \rangle$ and $\beta_{2,m} = \langle (\beta,0) , \phi_{2,m} \rangle$. Then:
\begin{equation}
\tilde{\mathcal{X}}(t)  = (\tilde{y}(t,\cdot),z(t,\cdot)) 
 = \sum_{n \geq 0} \tilde{x}_{1,n}(t) \psi_{1,n} + \sum_{m \geq 0} \tilde{x}_{2,m}(t) \psi_{2,m} . \label{eq: series expansion of the state bis}
\end{equation}
\begin{itemize}
\item For $n \geq 0$, using \eqref{eq: A phi_1_n}, the projection of \eqref{eq: abstract system homogeneous bis} onto $\psi_{1,n}$ gives \eqref{eq: dynamics tilde_x1n} and \eqref{eq: dynamics x1n} with $\nu_{1,n} = \alpha_{1,n} + \lambda_{1,n} \beta_{1,n}$.
\item For $m \notin \Delta_2(a,b)$, using \eqref{eq: A phi_2_m}, the projection of \eqref{eq: abstract system homogeneous bis} onto $\psi_{2,m}$ gives \eqref{eq: dynamics tilde_x2m} and \eqref{eq: dynamics x2m} with $\nu_{2,m} = \alpha_{2,m} + \lambda_{2,m} \beta_{2,m}$.
\item For $m\in\Delta_2(a,b)$. Let $n \geq 0$ be such that $(n,m)\in\Delta(a,b)$, meaning that $\lambda = \lambda_{1,n} = \lambda_{2,m}$. In that case, using \eqref{eq: phi generalized eigenvector}, the projection of \eqref{eq: abstract system homogeneous bis} onto $\psi_{1,n}$ and $\psi_{2,m}$ gives
\begin{equation}\label{eq: dynamics tilde_x1n tilde_x2m mult 2 bis}
	\begin{bmatrix} \dot{\tilde{x}}_{1,n} \\ \dot{\tilde{x}}_{2,m} \end{bmatrix}
	= \underbrace{\begin{bmatrix} \lambda & 0 \\ -\sqrt{2} s \mu_m & \lambda \end{bmatrix}}_{= M_{n}} \begin{bmatrix} \tilde{x}_{1,n} \\ \tilde{x}_{2,m} \end{bmatrix} + \begin{bmatrix} \alpha_{1,n} \\ \alpha_{2,m} \end{bmatrix} u + \begin{bmatrix} \beta_{1,n} \\ \beta_{2,m} \end{bmatrix} v
\end{equation}
yielding
\begin{equation}\label{eq: dynamics x1n x2m mult 2 bis}
	\begin{bmatrix} \dot{x}_{1,n} \\ \dot{x}_{2,m} \end{bmatrix}
	= M_{n} \begin{bmatrix} x_{1,n} \\ x_{2,m} \end{bmatrix} + \begin{bmatrix} \nu_{1,n} \\ \nu_{2,m} \end{bmatrix} u 
\end{equation}
where $\nu_{1,n} = \alpha_{1,n} + \lambda_{1,n} \beta_{1,n}$ and $\nu_{2,m} = \alpha_{2,m} + \lambda_{2,m} \beta_{2,m} - \sqrt{2}s\mu_m \beta_{1,n}$.
\end{itemize}

\subsubsection{Controllability of the modes}

The characterization of the controllabity of each mode of the PDE cascade \eqref{eq: system 3} requires the evaluation of the terms $\nu_{1,n}$ and $\nu_{2,m}$. Proceeding as in the proof of Lemma~\ref{lem: mu_i_k}, we obtain the following result.

\begin{lemma}
$\nu_{i,k} = - \phi_{i,k}^1(0)$ for all $k \geq 0$ and $i\in\{1,2\}$. 
\end{lemma}

We can now characterize the controllability of each mode of the PDE cascade \eqref{eq: system}:
\begin{itemize}
\item For $n \notin \Delta_1(a,b)$, the mode $\lambda_{1,n}$ has a one-dimensional dynamics given by \eqref{eq: dynamics x1n}, which is controllable because $\nu_{1,n} = - \phi_{1,n}^1(0) = - \sqrt{2} \neq 0$. 
\item For $m \notin \Delta_2(a,b)$, the mode $\lambda_{2,m}$ has a one-dimensional dynamics given by \eqref{eq: dynamics x2m}. It is controllable if and only if $\nu_{2,m} = - \phi_{2,m}^1(0) \neq 0$. If $b-a = m^2\pi^2$ then $\nu_{2,m} = s\mu_m \neq 0$, hence the mode is controllable. If $b-a \neq m^2\pi^2$, then the mode is controllable if and only if $\nu_{2,m} = - \frac{s\mu_m}{r_m}\tanh(r_m) \neq 0$, that is $\sinh(r_m) \neq 0$. Recalling that $r_m$ is one of the two square roots of $\lambda_{2,m}-a=b-a-m^2\pi^2 \neq 0$, the latter condition holds if and only if $m \notin \Theta'(a,b)$ with
\begin{equation}\label{eq: def Theta(a,b) bis}
		\Theta'(a,b) = \{  m\in\mathbb{N} \,\mid\, \exists k \geq 1  \;\mathrm{s.t.}  
		b-a = ( m^2 - k^2 ) \pi^2 \}. 
\end{equation}
\item When $(n,m)\in\Delta(a,b)$, we have $\lambda = \lambda_{1,n} = \lambda_{2,m}$. The two-dimensional dynamics \eqref{eq: dynamics x1n x2m mult 2 bis} is controllable because $s\mu_m \neq 0$ and $\nu_{1,n} = - \psi_{1,n}^1(0) = - \sqrt{2} \neq 0$.
\end{itemize}

\begin{remark}\label{rem: cont sys 4}
In the case of the PDE cascade \eqref{eq: system 4}, we apply the same approach with the change of variable  \eqref{eq: change of variable bis} done with the function $\varphi(x) = x^2$ so that $\varphi(0)=\varphi'(0)=0$ and $\varphi(1)=1$. The only difference in that $\nu_{i,k} = - (\phi_{i,k}^1)'(1)$. This implies that $\nu_{i,k} \neq 0$, meaning that all modes of the PDE cascade \eqref{eq: system 4} are always controllable. We reach here the same conclusion as for the two PDE cascades \eqref{eq: system} and \eqref{eq: system 2}. Indeed, while the collocated configuration (for the input and PDE interconnection) of the PDE cascade \eqref{eq: system 3} may lead to the loss of controllability for some modes and some specific reaction coefficients $a,b\in\mathbb{R}$, in the uncollocated setting of \eqref{eq: system 4} all modes are always controllable. 
\end{remark}

\subsection{State-feedback}

We proceed as in Section~\ref{sec: state-feedback} to design a state-feedback for the PDE cascade \eqref{eq: system 3}. Using the same definitions as the ones introduced in that section, the augmented model is controllable if and only if 
\begin{equation}\label{eq: controllability assumption bis}
\Theta'(a,b) \cap \{0,\ldots,N_{0}\} = \emptyset .
\end{equation} 
Therefore, the statement of Theorem~\ref{thm: state-feedback} also applies to the PDE cascade \eqref{eq: system 3} when replacing the controllability assumption \eqref{eq: controllability assumption} by \eqref{eq: controllability assumption bis}. The same approach applies to the PDE cascade \eqref{eq: system 4} when removing the controllability assumption \eqref{eq: controllability assumption bis}.

\subsection{Output-feedback}
In the case of the two PDE cascades \eqref{eq: system 3} and \eqref{eq: system 4}, since the control input applies to the $y$-equation, we consider a measurement performed on the $z$-equation. In the case of the distributed system output
\begin{equation}\label{eq: output bis}
	y_m(t) = \int_0^1 c(x) z(t,x) \,\mathrm{d}x
\end{equation}
for some $c \in L^2(0,1)$, we define $c_{1,n} = \int_0^1 c(x) \psi_{1,n}^2(x) \,\mathrm{d}x$ and $c_{2,m} = \int_0^1 c(x) \psi_{2,m}^2(x) \,\mathrm{d}x$. In the case of the pointwise measurement
\begin{equation}\label{eq: Dirichlet measurement bis}
	y_m(t) = z(t,\xi_p) 
\end{equation}
or 
\begin{equation}\label{eq: Neumann measurement bis}
	y_m(t) = \partial_x z(t,\xi_p) 
\end{equation}
for some $\xi_p\in[0,1]$, we define in the first case $c_{1,n} = \psi_{1,n}^2(\xi_p)$ and $c_{2,m} = \psi_{2,m}^2(\xi_p)$, and in the second case, $c_{1,n} = (\psi_{1,n}^2)'(\xi_p)$ and $c_{2,m} = (\psi_{2,m}^2)'(\xi_p)$. For the three measurement settings, we infer from \eqref{eq: series expansion of the state bis} that the series expansion \eqref{eq: output series expansion} holds. Adopting the notations introduced therein, the output-feedback control strategy designed in Section~\ref{sec: output-feedback} applies, provided that the pair $(A_1,C_1)$ is observable. The Hautus test shows that this observability condition holds if and only if 
\begin{subequations}\label{eq: obsv condition bis}
\begin{align}
& c_{1,n} \neq 0 , \quad  &&\forall n\in\{0,1,\ldots,N_0\}\backslash\Delta_1(a,b) , \\ 
& c_{2,m} \neq 0 , \quad &&\forall 0 \leq m \leq M_{0} .
\end{align}
\end{subequations}
This shows that the statements of Theorems~\ref{thm2} and~\ref{thm3} apply to the PDE cascade \eqref{eq: system 3} when replacing the controllability assumption \eqref{eq: controllability assumption} by \eqref{eq: controllability assumption bis} as well as the observability assumption \eqref{eq: obsv condition} by \eqref{eq: obsv condition bis}. The same approach applies to the PDE cascade \eqref{eq: system 4} when removing from the theorem statement the controllability assumption \eqref{eq: controllability assumption bis}.

\begin{remark}
For the pointwise measurement \eqref{eq: Neumann measurement bis}, we note that $\psi_{2,0}^2 = 1$, hence $c_{2,0} = 0$. Therefore, the mode $\lambda_{2,0} = b$ is never observable. Hence, in our approach, this measurement cannot be used to stabilize the $z$-part of the PDE cascades \eqref{eq: system 3} and \eqref{eq: system 4}. Our approach can however be applied as soon as the $z$-part of the cascade is open-loop stable, that is, when $\lambda_{2,0}=b < 0$. In this case, the modes $\lambda_{1,n} \geq 0$ associated with the $y$-part of the cascade, that fail to be exponentially stable, can be stabilized using the pointwise measurement \eqref{eq: Neumann measurement bis} provided that $c_{1,n} \neq 0$.
\end{remark}

\section{Conclusion}\label{sec: conclusion}
In this article we have studied the problem of state or output feedback stabilization for various heat-heat PDE cascades. Thanks to a detailed spectral study, we have established the Riesz basis property for the generalized eigenvectors of the underlying unbounded operators. This has allowed us to fully characterize the controllability properties of the four PDE cascades, and to derive explicit state-feedback and output-feedback control strategies. Other kinds of PDE cascades, possibly in higher dimension, will be studied in future works.

\appendix
\section{Appendix: Exact controllability properties for the PDE system \eqref{eq: system}}\label{appendix}
In this section, we show how to use the spectral properties derived in Section~\ref{sec: spectral analysis} to establish the exact null controllability property of \eqref{eq: system} in an appropriate Hilbert space $V_0 \subset \mathcal{H}^0$ of dual $V_0'$ with respect to the pivot space $\mathcal{H}^0$. This is achieved similarly to what has been done in~\cite{lhachemi2025controllability} for a wave-heat cascade by leveraging the Riesz spectral properties of the system. More precisely, we establish a finite-time observability inequality for the dual system $\dot{\mathcal{X}}(t) = \mathcal{A}^* \mathcal{X}(t)$, where $\mathcal{X} = (\mathcal{X}^1,\mathcal{X}^2)$, with observation $\mathcal{B}_0^*\mathcal{X}(t) = \mathcal{X}^2(t,0)$, i.e., for any $T > 0$ there exists $C_T^0 > 0$ such that 
$$
\int_0^T \vert  \mathcal{X}^2(t,0) \vert^2 \,\mathrm{d}t \geq C_T^0 \Vert \mathcal{X}(T,\cdot) \Vert_{V_0'}^2
$$
for all solutions of $\dot{\mathcal{X}}(t) = \mathcal{A}^* \mathcal{X}(t)$. 
Integrating, we have
\begin{multline*}
\mathcal{X}^2(t,0)
= \sum_{n\in\mathbb{N}\backslash\Delta_1(a,b)} e^{\lambda_{1,n} t} \langle \mathcal{X}(0) , \phi_{1,n} \rangle \psi_{1,n}^2(0) 
+ \sum_{m\in\mathbb{N}\backslash\Delta_2(a,b)} e^{\lambda_{2,m} t} \langle \mathcal{X}(0) , \phi_{2,m} \rangle \psi_{2,m}^2(0) \\
+ \sum_{(n,m)\in\Delta(a,b)} \begin{bmatrix} \psi_{1,n}^2(0) & \psi_{2,m}^2(0) \end{bmatrix} e^{M_n^\top t} \begin{bmatrix} \langle \mathcal{X}(0) , \phi_{1,n} \rangle \\ \langle \mathcal{X}(0) , \phi_{2,m} \rangle \end{bmatrix}
\end{multline*}
where we recall that $\mathrm{Card}(\Delta(a,b)) < \infty$. Therefore, using the matrix $M_n$ defined in \eqref{eq: dynamics tilde_x1n tilde_x2m mult 2} with $s \neq 0$ and recalling that $\psi_{2,m}^2(0) = \mu_m \neq 0$, a necessary condition for \eqref{eq: system} to be exactly controllable is, based on \eqref{eq: psi_1_n^2 - 1}, that $\psi_{1,n}^2(0) = - \frac{\sqrt{2}s}{r_n^*} \coth(r_n^*) \neq 0$ for all integers $n \in \mathbb{N}\backslash\Delta_1(a,b)$. Since $\Delta_1(a,b) \cap \Theta(a,b) = \emptyset$, this is equivalent to $\Theta(a,b) = \emptyset$. Conversely, assume that $\Theta(a,b) = \emptyset$. Then we infer that\footnote{$\Theta(a,b) = \emptyset$ implies $a \neq b$, see Remark~\ref{rmk: discussion Theta(a,b)}.} $\psi_{1,n}^2(0) \sim \frac{\sqrt{2}s(a-b)}{2n^2\pi^2}$  as $n \rightarrow +\infty$ and $\psi_{2,m}(0) \sim \sqrt{2}$ as $m \rightarrow +\infty$. Due to the possible occurrence of a finite number of blocs $M_n$ that are similar to Jordan blocks, we cannot directly apply the M{\"u}ntz-Sz{\'a}sz theorem (see~\cite{avdoninivanov}). But, using~\cite[Thm~V.4.16, p.94]{boyer2022controllability}, the moment method (see~\cite{trelat2024control}) yields the claimed observability property, and thus by duality, the exact null controllability property of \eqref{eq: system} in the Hilbert space 
\begin{equation*}
V_0 = \Big\{ \sum_{n\in\mathbb{N}} a_n \phi_{1,n} + \sum_{m\in\mathbb{N}} b_m \phi_{2,m} \,\mid\, a_n,b_n\in\mathbb{R} , \ \ 
\sum_{n\in\mathbb{N}} \Big( 1 + \frac{n^4}{s^2(a-b)^2} \Big) \vert a_n \vert^2 + \sum_{m\in\mathbb{N}} \vert b_m \vert^2  < +\infty \Big\} .
\end{equation*}
The exact controllability property at time $T > 0$ is established similarly in the Hilbert space
\begin{align*}
V = \Big\{ & \sum_{n\in\mathbb{N}} a_n \phi_{1,n} + \sum_{m\in\mathbb{N}} b_m \phi_{2,m} \,\mid\, a_n,b_n\in\mathbb{R} , \\
& \qquad \sum_{n\in\mathbb{N}} \Big( 1 + \frac{n^4}{s^2(a-b)^2} \Big) e^{2T(n+1/2)^2\pi^2} \vert a_n \vert^2 + \sum_{m\in\mathbb{N}} \vert b_m \vert^2 e^{2Tm^2\pi^2}  < +\infty \Big\} .
\end{align*}
In both cases, we observe the ``collapse'' of the controllability spaces $V_0$ and $V$ when we either uncouple the two equations ($s=0$) or when $a=b$, consistently with Lemmas~\ref{lem: eigenelements A} and~\ref{lem: cont finite dim model}. 
 \\[2mm]
Using the same approach, we can also study the final state observability of the system \eqref{eq: system} with measurement \eqref{eq: output}. A necessary condition is that $c_{1,n} \neq 0$ and $c_{2,m} \neq 0$ for all $n \in \mathbb{N}$ and $m \in\mathbb{N}\backslash\Delta_2(a,b)$. This condition appears to be sufficient in the Hilbert space 
\begin{equation*}
W = \Big\{ \sum_{n\in\mathbb{N}} a_n \phi_{1,n} + \sum_{m\in\mathbb{N}} b_m \phi_{2,m} \,\mid\, a_n,b_n\in\mathbb{R} ,\ \ 
\sum_{n\in\mathbb{N}} c_{1,n}^2 \vert a_n \vert^2 + \sum_{m\in\mathbb{N}} c_{2,m}^2 \vert b_m \vert^2  < +\infty \Big\} .
\end{equation*}

\paragraph{Acknowlegment.}
The third author acknowledges the support of ANR-20-CE40-0009 (TRECOS).

\bibliographystyle{abbrv}        
\bibliography{autosam}           

\begin{thebibliography}{10}

\bibitem{ammar2011recent}
F.~Ammar-Khodja, A.~Benabdallah, M.~Gonz{\'a}lez-Burgos, and L.~De~Teresa.
\newblock Recent results on the controllability of linear coupled parabolic
  problems: a survey.
\newblock {\em Math. Control Relat. Fields}, 1(3):267--306, 2011.

\bibitem{avdoninivanov}
S.~A. Avdonin and S.~A. Ivanov.
\newblock {\em Families of exponentials}.
\newblock Cambridge University Press, Cambridge, 1995.
\newblock The method of moments in controllability problems for distributed
  parameter systems, Translated from the Russian and revised by the authors.

\bibitem{benabdallah2020block}
A.~Benabdallah, F.~Boyer, and M.~Morancey.
\newblock A block moment method to handle spectral condensation phenomenon in
  parabolic control problems.
\newblock {\em Annales Henri Lebesgue}, 3:717--793, 2020.

\bibitem{bhandari2021boundary}
K.~Bhandari and F.~Boyer.
\newblock Boundary null-controllability of coupled parabolic systems with
  {R}obin conditions.
\newblock {\em Evolution Equations and Control Theory}, 10(1):61--102, 2021.

\bibitem{boyer2022controllability}
F.~Boyer.
\newblock Controllability of linear parabolic equations and systems.
\newblock {\em hal-02470625v4f}, 2022.

\bibitem{chen2017backstepping}
S.~Chen, R.~Vazquez, and M.~Krstic.
\newblock Backstepping control design for a coupled hyperbolic-parabolic mixed
  class {PDE} system.
\newblock In {\em IEEE 56th Annual Conference on Decision and Control}, pages
  664--669, 2017.

\bibitem{chowdhury2023boundary}
S.~Chowdhury, R.~Dutta, and S.~Majumdar.
\newblock Boundary controllability and stabilizability of a coupled first-order
  hyperbolic-elliptic system.
\newblock {\em Evolution Equations \& Control Theory}, 12(3), 2023.

\bibitem{coron2004global}
J.-M. Coron and E.~Tr{\'e}lat.
\newblock Global steady-state controllability of one-dimensional semilinear
  heat equations.
\newblock {\em SIAM Journal on Control and Optimization}, 43(2):549--569, 2004.

\bibitem{curtain2012introduction}
R.~F. Curtain and H.~Zwart.
\newblock {\em An introduction to infinite-dimensional linear systems theory},
  volume~21.
\newblock Springer Science \& Business Media, 2012.

\bibitem{ghousein2020backstepping}
M.~Ghousein and E.~Witrant.
\newblock Backstepping control for a class of coupled hyperbolic-parabolic
  {PDE} systems.
\newblock In {\em 2020 American Control Conference (ACC)}, pages 1600--1605.
  IEEE, 2020.

\bibitem{gohberg1978introduction}
I.~Gohberg and M.~G. Kreuin.
\newblock {\em Introduction to the theory of linear nonselfadjoint operators},
  volume~18.
\newblock American Mathematical Soc., 1978.

\bibitem{grune2021finite}
L.~Gr{\"u}ne and T.~Meurer.
\newblock Finite-dimensional output stabilization of linear diffusion-reaction
  systems--a small-gain approach.
\newblock {\em arXiv preprint arXiv:2104.06102}, 2021.

\bibitem{kang2016stabilisation}
W.~Kang and B.-Z. Guo.
\newblock Stabilisation of unstable cascaded heat partial differential equation
  system subject to boundary disturbance.
\newblock {\em IET Control Theory \& Applications}, 10(9):1027--1039, 2016.

\bibitem{katz2020constructive}
R.~Katz and E.~Fridman.
\newblock Constructive method for finite-dimensional observer-based control of
  {1-D} parabolic {PDEs}.
\newblock {\em Automatica}, 122:109285, 2020.

\bibitem{khodja2016new}
F.~A. Khodja, A.~Benabdallah, M.~Gonz{\'a}lez-Burgos, and L.~de~Teresa.
\newblock New phenomena for the null controllability of parabolic systems:
  Minimal time and geometrical dependence.
\newblock {\em Journal of Mathematical Analysis and Applications},
  444(2):1071--1113, 2016.

\bibitem{lhachemi2020finite}
H.~Lhachemi and C.~Prieur.
\newblock Finite-dimensional observer-based boundary stabilization of
  reaction-diffusion equations with either a {D}irichlet or {N}eumann boundary
  measurement.
\newblock {\em Automatica}, 135:109955, 2022.

\bibitem{lhachemi2021nonlinear}
H.~Lhachemi and C.~Prieur.
\newblock Nonlinear boundary output feedback stabilization of
  reaction--diffusion equations.
\newblock {\em Systems \& Control Letters}, 166:105301, 2022.

\bibitem{lhachemi2020pi}
H.~Lhachemi, C.~Prieur, and E.~Tr{\'e}lat.
\newblock {PI} regulation of a reaction--diffusion equation with delayed
  boundary control.
\newblock {\em IEEE Transactions on Automatic Control}, 66(4):1573--1587, 2020.

\bibitem{lhachemi2025controllability}
H.~Lhachemi, C.~Prieur, and E.~Tr{\'e}lat.
\newblock Controllability and stabilization of a wave-heat cascade system.
\newblock {\em Under review}, 2025.

\bibitem{rosier2013unique}
L.~Rosier and B.-Y. Zhang.
\newblock Unique continuation property and control for the
  {B}enjamin--{B}ona--{M}ahony equation on a periodic domain.
\newblock {\em Journal of Differential Equations}, 254(1):141--178, 2013.

\bibitem{russell1978controllability}
D.~L. Russell.
\newblock Controllability and stabilizability theory for linear partial
  differential equations: recent progress and open questions.
\newblock {\em {SIAM} Review}, 20(4):639--739, 1978.

\bibitem{sakawa1983feedback}
Y.~Sakawa.
\newblock Feedback stabilization of linear diffusion systems.
\newblock {\em SIAM Journal on Control and Optimization}, 21(5):667--676, 1983.

\bibitem{tang2024boundary}
J.-Q. Tang, J.-M. Wang, and W.~Kang.
\newblock Boundary feedback stabilization of an unstable cascaded heat--heat
  system with different reaction coefficients.
\newblock {\em Systems \& Control Letters}, 183:105684, 2024.

\bibitem{tang2025sampled}
J.-Q. Tang, J.-M. Wang, and W.~Kang.
\newblock Sampled-data control of an unstable cascaded heat--heat system with
  different reaction coefficients.
\newblock {\em Automatica}, 171:111904, 2025.

\bibitem{trelat2024control}
E.~Tr{\'e}lat.
\newblock {\em Control in finite and infinite dimension}.
\newblock Springer, 2024.

\bibitem{wang2015stabilization}
J.-M. Wang, L.-L. Su, and H.-X. Li.
\newblock Stabilization of an unstable reaction--diffusion {PDE} cascaded with
  a heat equation.
\newblock {\em Systems \& Control Letters}, 76:8--18, 2015.

\bibitem{zhang2004polynomial}
X.~Zhang and E.~Zuazua.
\newblock Polynomial decay and control of a {1-D} hyperbolic--parabolic coupled
  system.
\newblock {\em Journal of Differential Equations}, 204(2):380--438, 2004.

\end{thebibliography}

\end{document}